\documentclass[11pt]{amsart}
\usepackage{amssymb,amscd,amsthm,verbatim}
\usepackage{amsbsy}
\usepackage{latexsym}
\usepackage[all,cmtip]{xy}
\usepackage[mathscr]{euscript}
\usepackage{bbm}
\usepackage{stmaryrd}
\usepackage[left=1in,right=1in,top=1in,bottom=1in]{geometry}

\date{\today}

\newcommand{\bbC}{{\mathbbm{C}}}
\newcommand{\bbD}{{\mathbb{D}}}
\newcommand{\bbJ}{{\mathbbm{J}}}

\newcommand{\bbR}{{\mathbbm{R}}}

\newcommand{\bbT}{{\mathbbm{T}}}
\newcommand{\bbZ}{{\mathbbm{Z}}}

\newcommand{\CP}{{\mathcal{P}}}

\newcommand{\Hi}{{\mathcal{H}}}
\newcommand{\Ran}{{\mathrm{ran}}}

\newcommand{\sJ}{{\mathscr{J}}}

\newcommand{\SL}{{\mathrm{SL}}}

\newcommand{\fp}{{\mathfrak{p}}}

\newcommand{\ac}{{\mathrm{ac}}}
\renewcommand{\sc}{{\mathrm{sc}}}
\newcommand{\pp}{{\mathrm{pp}}}

\newcommand{\LP}{{\mathrm{LP}}}
\newcommand{\PT}{{\mathrm{PT}}}

\newcommand{\Hd}{{\mathrm{H}}}

\newcommand{\hull}{{\mathrm{hull}}}

\newcommand{\Leb}{{\mathrm{Leb}}}

\newcommand{\orb}{{\mathrm{orb}}}

\newcommand{\per}{{\mathrm{per}}}

\newcommand{\tr}{{\mathrm{Tr}}}

\newtheorem{theorem}{Theorem}[section]
\newtheorem{lemma}[theorem]{Lemma}
\newtheorem{prop}[theorem]{Proposition}
\newtheorem{coro}[theorem]{Corollary}

\theoremstyle{definition}

\newtheorem{question}[theorem]{Question}
\newtheorem*{definition}{Definition}
\newtheorem{example}[theorem]{Example}
\newtheorem{remark}[theorem]{Remark}
\theoremstyle{plain}

\allowdisplaybreaks \numberwithin{equation}{section}

\DeclareMathOperator{\supp}{supp}
\DeclareMathOperator*{\slim}{s-lim}

\newcommand{\set}[1]{\left\{#1\right\}}

\newcommand{\defeq}{{:=}}

\begin{document}

\title[Limit-Periodic Operators]{Spectral Properties of Limit-Periodic Operators}

\begin{abstract}
We survey results concerning the spectral properties of limit-periodic operators. The main focus is on discrete one-dimensional Schr\"odinger operators, but other classes of operators, such as Jacobi and CMV matrices, continuum Schr\"odinger operators and multi-dimensional Schr\"odinger operators, are discussed as well.

We explain that each basic spectral type occurs, and it does so for a dense set of limit-periodic potentials. The spectrum has a strong tendency to be a Cantor set, but there are also cases where the spectrum has no gaps at all. The possible regularity properties of the integrated density of states range from extremely irregular to extremely regular. Additionally, we present background about periodic Schr\"odinger operators and almost-periodic sequences.

In many cases we outline the proofs of the results we present.
\end{abstract}

\author[D.\ Damanik]{David Damanik}
\address{Department of Mathematics, Rice University, Houston, TX~77005, USA}
\email{damanik@rice.edu}
\thanks{D.D.\ was supported in part by NSF grants DMS--1361625 and DMS--1700131.}

\author[J.\ Fillman]{Jake Fillman}
\address{Department of Mathematics, Virginia Polytechnic Institute and State University, 225 Stanger Street, Blacksburg, VA~24061, USA}
\email{fillman@vt.edu}
\thanks{J.F.\ was supported in part by an AMS-Simons travel grant, 2016--2018}

\maketitle

\setcounter{tocdepth}{1}
\tableofcontents

\section{Introduction}

We will talk about Schr\"odinger operators $H$ acting in $\ell^2(\bbZ)$ via
\[
[Hu](n)
=
u(n-1) + u(n+1) + V(n) u(n)
\]
with periodic and limit-periodic potentials. We say that $V$ (or, interchangeably, $H$) is \emph{periodic} if there is some $q \in \bbZ_+$ with $V(n+q) = V(n)$ for every $n$. We say that $V$ (or $H$) is \emph{limit-periodic} if $V$ lies in the closure of the space of periodic potentials, that is, if there exist $V_1, V_2, \ldots$ periodic with
\[
\lim_{n\to\infty} \| V - V_n\|_\infty
=
0.
\]

One might think that such operators would be rather similar to periodic operators, and that many characteristics and results can be ``pushed through'' to the limit using uniform convergence. However, one would be very wrong! Periodic operators always exhibit purely absolutely continuous spectrum of multiplicity two supported on a spectrum which is a finite union of nondegenerate closed intervals, and the associated quantum dynamics are ballistic (i.e.\ a particle subjected to a periodic potential behaves like a free particle). Every single one of these characteristics can be broken in the limit-periodic regime. Concretely, in the limit-periodic regime:
\begin{itemize}
\item The spectrum can be a Cantor set (a dense set of gaps). In fact, the spectrum may have zero Hausdorff dimension;
\item The spectral type may be purely singular continuous or even pure point;
\item The quantum dynamics may be localized in a very strong sense.
\end{itemize}

The study of these operators was begun in the 1980s by work of  Avron-Simon \cite{AS81}, Moser \cite{M81}, and the Soviet school \cite{Chul81, Chul84, egorova, MC84, PT1, PT2}. These classical results focused on the presence of Cantor spectrum and purely absolutely continuous spectral measures. A notable exception is the work \cite{P83} by P\"oschel from that era, which exhibited examples that are uniformly localized. There has been a recent uptick in activity, ushered in by Avila \cite{A09}, which has led to the realization that new phenomena are possible, such as the genericity of purely singular continuous spectrum \cite{A09, DG11a} and the denseness of pure point spectrum \cite{DG16} in the space of all limit-periodic potentials. Another new phenomenon established in a recent paper is the Lipschitz continuity of the integrated density of states in P\"oschel's examples \cite{DF2018}.

The paper is structured as follows. We provide a quick introduction to the periodic case in Section~\ref{s.periodic}. Almost periodic sequences and their hulls are discussed in Section~\ref{s.almostperiodic}. Every limit-periodic sequence is almost-periodic, and those almost-periodic sequences that are limit-periodic can be characterized via a topological property of their hulls: they need to be totally disconnected. Moreover, this discussion also shows that the study of almost-periodic operators fits in the broader class of ergodic operators, which are discussed in Section~\ref{s.ergodic}, along with the standard quantities associated with them: the Lyapunov exponent and the integrated density of states. Section~\ref{s.ac} contains results about limit-periodic operators with purely absolutely continuous spectral measures, supported on thick (Carleson homogeneous) Cantor sets. The genericity of zero-measure Cantor spectrum and purely singular continuous spectral measures is the subject of Section~\ref{s.sc}. Limit-periodic operators with pure point spectrum are presented in Section~\ref{s.pp}, both those that arise in the large coupling regime and those that arise due to local randomness. Section~\ref{s.ids} addresses regularity results for the integrated density of states. While we focus on discrete Schr\"odinger operators in this paper, much of what we say has analogs for related families of operators in mathematical physics, namely: Jacobi matrices, CMV matrices/quantum walks, and continuum Schr\"odinger operators. We describe these operators and some of the known results in Section~\ref{s.other}. Motivated by the known results, we present some interesting open problems in Section~\ref{sec:problems}. Finally, since the hull of a limit-periodic potential is a totally disconnected group, we discuss some characteristics of totally disconnected groups in Appendix~\ref{s.profinite}.

\medskip

\noindent\textbf{Acknowledgment.} Some of this material was presented in a lecture series at the Academy of Mathematics and Systems Science (Chinese Academy of Sciences), Beijing in February 2018. D.D.\ would like to express his gratitude to Zhe Zhou and AMSS/CAS for the kind invitation and extraordinary hospitality and all the lecture series attendees for their interest and participation.

\section{A Crash Course on Periodic Operators}\label{s.periodic}
In the present section, we will review some aspects of the spectral analysis of periodic operators in 1D. Throughout, assume that $V$ is $q$-periodic and that $q \in \bbZ_+$ is chosen to be minimal with respect to this property. We will at least sketch most of the arguments here. With the exception of the discussion of ballistic motion in Section~\ref{ss:ballistic}, the results in this section are textbook material; lucid references include \cite[Chapter~5]{simszego} and \cite[Chapter~7]{teschljacobi}.

The eigenvalue equation for $H$ reads
\begin{equation} \label{eq:diffEq1}
u(n-1) + u(n+1) + V(n) u(n)
=
E u(n)
\end{equation}
where $E$ is a scalar. We shall see presently that this equation has no (nontrivial) solution $u \in \ell^2(\bbZ)$ when $V$ is periodic, so we will need to be somewhat more generous in our interpretation of \eqref{eq:diffEq1}. Namely, we will consider \eqref{eq:diffEq1} for arbitrary $u \in \bbC^\bbZ$, not just $u \in \ell^2(\bbZ)$. Through a standard abuse of notation, we will also write $H u$ for $u \in \bbC^\bbZ$, where $[H u](n)$ is simply defined by the left-hand side of \eqref{eq:diffEq1}.

\subsection{Floquet Theory} \label{ss:floq}

The difference equation \eqref{eq:diffEq1} can be rewritten as a $2\times 2$ matrix recursion:
\begin{equation}\label{eq:diffEq2}
\begin{bmatrix}
u(n+1) \\ u(n)
\end{bmatrix}
=
\begin{bmatrix}
E - V(n) & -1 \\
1 & 0
\end{bmatrix}
\begin{bmatrix}
u(n) \\ u(n-1)
\end{bmatrix}.
\end{equation}
Clearly then, the asymptotic characteristics of $u$ can be encoded by the \emph{monodromy matrix}
\[
\Phi(E)
=
\begin{bmatrix}
E - V(q) & -1 \\
1 & 0
\end{bmatrix}
\begin{bmatrix}
E - V(q-1) & -1 \\
1 & 0
\end{bmatrix}
\cdots
\begin{bmatrix}
E - V(1) & -1 \\
1 & 0
\end{bmatrix}
\]
and the relationship between $(u(1), u(0))^\top$ and the eigenvector(s) of $\Phi(E)$. Since $\Phi \in \SL(2,\bbC)$, the character of its eigenvectors is entirely encoded by its trace. We may characterize the spectrum of $H$ as follows:

\begin{theorem} \label{t:floquet}
If $V$ is $q$-periodic with discriminant $D$, then
\[
\sigma(H_V)
=
\set{E \in \bbC : D(E) \in [-2,2]}
=
\set{E \in \bbR : |D(E)| \leq 2}.
\]
For all $c \in (-2,2)$, all solutions of $D(E) = c$ are real and simple. For $c = \pm2$, all solutions of $D(E) = c$ are real and at most doubly degenerate.
\end{theorem}

\begin{proof}[Proof Sketch]

Since $H_V$ is self-adjoint, we focus on real $E$. For each $E \in \bbR$, there are four possibilities for $\Phi(E)$:
\medskip

\begin{enumerate}
\item $|\tr\,\Phi(E)| < 2$: Then, the powers of $\Phi(E)$ are uniformly bounded, so all solutions of \eqref{eq:diffEq1} are bounded;
\medskip

\item $|\tr\,\Phi(E)| = 2$ and $\Phi(E) \notin \{\pm I\}$: There is a Jordan anomaly. One solution of \eqref{eq:diffEq1} is bounded, and any linearly independent solution grows linearly;
\medskip

\item $\Phi(E) = \pm I$. Clearly, all solutions of \eqref{eq:diffEq1} are bounded;
\medskip
\item $|\tr\, \Phi(E)| > 2$. There are linearly independent solutions $u_+$ and $u_-$ of \eqref{eq:diffEq1}  with the property that $u_\pm$ decays exponentially at $\pm \infty$ and grows exponentially at $\mp \infty$.
\end{enumerate}

In cases (1), (2), and (3), one can use a bounded solution to construct a Weyl sequence, and hence $E \in \sigma(H)$. In case (4), we have two special solutions $u_\pm$, defined as follows: let $(u_+(1), u_+(0))^\top$ be a contracting eigenvector of $\Phi(E)$ and let $(u_-(1), u_-(0))^\top$ be an expanding eigenvector of $\Phi(E)$, and use \eqref{eq:diffEq2} to extend these initial data to produce solutions of \eqref{eq:diffEq1}. Then, $u_\pm$ decays exponentially at $\pm \infty$ and one can use these solutions to construct a Green function for $H$. Concretely, one can check that
\[
\langle \delta_n, R \delta_m\rangle
=
\frac{u_- (n \wedge m) u_+(n \vee m)}{u_-(0)u_+(1) - u_+(0)u_-(1)}
\]
($n \wedge m = \min(n,m)$ and $n \vee m = \max(n,m)$) defines a bounded operator with $R(H-E) = (H-E)R =I$ by Schur's criterion.

In view of the discussion above, it follows that any solution of $D(E) = c$ with $c\in [-2,2]$ must be real, since such an $E$ belongs to the spectrum of $H$ and $H$ is self-adjoint. The multiplicity statements follow from Rouch\'e's theorem. Concretely, if $D(E) - c$ has a double (or higher) root at $E \in \bbR$ for some $c \in (-2,2)$, then Rouch\'e's theorem implies the existence of non-real $z$ near $E$ such that $D(z) \in (-2,2)$, contradicting reality of such solutions. Similarly, if $D(E)-c$ vanishes to third order or higher for some choice of $c \in \{\pm2\}$, then Rouch\'e's theorem implies the existence of nearby non-real $z$ with $D(z) \in [-2,2]$.
\end{proof}

Thus, we see that the spectrum of $H$ consists of $q$ nondegenerate closed intervals, which we call the \emph{bands}:
\[
\sigma(H)
=
\bigcup_{j=1}^q [\alpha_j,\beta_j].
\]
Here, $\alpha_1 < \beta_1 \leq \alpha_2 < \cdots \leq \alpha_q < \beta_q$ denotes an enumeration of the solutions of $D(E) = \pm 2$, counted with multiplicity. Each of the closed intervals $[\alpha_j,\beta_j]$ may be obtained as the closure of a connected component of $D^{-1}((-2,2))$. In the event that $\beta_j < \alpha_{j+1}$, we say that the $j$th \emph{gap} is open. Otherwise, we say that we have a \emph{closed gap} at the point $\alpha_j = \beta_{j+1}$. Notice that one has a closed gap at $c$ if and only if $D^2-4$ vanishes to second order at $c$, that is, $D(c) = \pm 2$ and $D'(c) = 0$.

\subsection{Bloch Waves} \label{ss:bloch}
If $E \in \sigma(H)$, then $D(E) \in [-2,2]$. Writing $D(E) = 2\cos\theta$ for some $\theta \in [0,\pi]$, we see that $\Phi(E)$ has eigenvalues $e^{\pm i\theta}$. Using an eigenvector corresponding to the eigenvalue $e^{i\theta}$ as an initial condition, we may generate a solution $u_+$ to $Hu = Eu$ so that
\begin{equation} \label{eq:floquetCond}
u_+(n+q)
=
e^{i \theta} u_+(n).
\end{equation}
Similarly, using the eigenvector for $e^{-i\theta}$ as an initial condition, one obtains a solution $u_-$ satisfying $u_-(n+q)= e^{-i\theta}u_-(n)$. On can encode solutions of $Hu=Eu$ satisfying a condition like \eqref{eq:floquetCond} in terms of $q\times q$ matrices with suitable boundary conditions. Concretely, define
\[
H_{\theta}
=
H_{\theta,q}
=
\begin{bmatrix}
V_1 & 1 &&& e^{-i\theta} \\
1 & V_2 & 1 \\
 & \ddots & \ddots & \ddots \\
 && 1 & V_{q-1} & 1 \\
 e^{i\theta} &&& 1 & V_q
\end{bmatrix}.
\]

\begin{theorem} \label{t:bloch}
If $V$ is $q$-periodic, we have
\begin{equation} \label{eq:blochSpecChar}
\sigma(H)
=
\bigcup_{\theta \in [0,\pi]} \sigma(H_\theta).
\end{equation}
Moreover, the discriminant and the characteristic polynomials are related via
\begin{equation} \label{eq:discriminantCharPoly}
\det(E - H_\theta)
=
D(E) - 2\cos\theta.
\end{equation}
\end{theorem}

\begin{proof}[Proof Sketch]
Given $E \in \sigma(H)$, write $D(E) = 2\cos\theta$ as before, and define $u_\pm$ as above.
One can then check that $(u_+(1), u_+(2), \ldots, u_+(q))^\top$ is an eigenvector of $H_\theta$ with eigenvalue $E$. On the other hand, if $E$ is an eigenvalue of $H_\theta$, an eigenvector of $H_\theta$ can be extended to a solution of $Hu = Eu$ satisfying \eqref{eq:floquetCond}. Thus, $e^{i\theta}$ is an eigenvalue of $\Phi(E)$, which implies $D(E) = 2\cos\theta$. This discussion proves \eqref{eq:blochSpecChar}.

Thus, for fixed $\theta$, $D(E) - 2\cos(\theta)$ and $\det(E - H_\theta)$ are monic polynomials in $E$ with the same degree and the same roots. For $\theta \in (0,\pi)$, all roots of $D(E) - 2\cos\theta$ are simple by Theorem~\ref{t:floquet}, so \eqref{eq:discriminantCharPoly} follows for all $E \in \bbC$ and all $\theta \in (0,\pi)$. Then, fixing $E$, both sides of \eqref{eq:discriminantCharPoly} define entire functions of $\theta$ that agree for $\theta \in (0,\pi)$ and hence must agree everywhere.
\end{proof}

Let us note a couple of pleasant consequences of the foregoing result. First of all, since $D(E) - c$ can only vanish to first order for $c \in (-2,2)$, it follows from \eqref{eq:discriminantCharPoly} that $H_\theta$ has simple spectrum for $\theta \in (0,\pi)$. One can also prove this directly (cf.\ \cite[Section~5.3]{simszego}). We can also read off from \eqref{eq:discriminantCharPoly} that the spectrum of $H$ has a closed gap at energy $E_0$ with $D(E_0) = 2$ (resp.\ $D(E_0) = -2$) if and only if all solutions to $Hu = E_0 u$ are periodic (resp.\ anti-periodic).
\bigskip

One can formulate Theorem~\ref{t:bloch} in terms of direct integrals in a natural manner. Concretely, consider the space
\begin{align*}
\mathcal{H}_q
& \defeq
L^2([0,2\pi), \bbC^q) \\
& =
\set{ \vec{f}:[0,2\pi) \to \bbC^q : \int_0^{2\pi} \|\vec{f}(\theta)\|_{\bbC^q}^2 < \infty }.
\end{align*}
Naturally, $\Hi_q$ is a Hilbert space with respect to the inner product
\[
\langle f, g \rangle_{\Hi_q}
=
\int_0^{2\pi} \langle f(\theta), g(\theta) \rangle_{\bbC^q} \, \frac{d\theta}{2\pi}.
\]
This space is sometimes written as a direct integral, especially in the physics literature:
\[
\mathcal{H}_q
=
\int^\oplus_{[0,2\pi)} \bbC^q \frac{d\theta}{2\pi}.
\]
We may define a linear operator $\mathcal F: \ell^2(\bbZ) \to \Hi_q$ by $\mathcal F: u \mapsto \widehat u$, where
\begin{equation} \label{eq:modqFourierDef}
\widehat u_j(\theta)
=
\sum_{\ell \in \bbZ} u_{j + \ell q} e^{-i\ell\theta},
\quad
\theta \in [0,2\pi), \; 1 \le j \le q.
\end{equation}
Of course, $\mathcal F$ is initially only defined on $\ell^1(\bbZ)$, but it admits a unique extension to a unitary operator on $\ell^2(\bbZ)$ by Parseval's formula.

\begin{lemma}
The map $\mathcal F$ extends to a unitary operator from $\ell^2(\bbZ)$ to $\Hi_q$; its inverse maps $f \in \Hi_q$ to $\check f \in \ell^2(\bbZ)$, defined by
\[
\check f_{j + \ell q}
=
\int_0^{2\pi} e^{i\ell \theta} f_j(\theta) \, \frac{d\theta}{2\pi},
\quad 1 \le j \le q, \; \ell \in \bbZ.
\]
\end{lemma}

\begin{proof}
For each $1 \leq j \leq q$ and each $\ell \in \bbZ$, we define $\varphi_{j,\ell} \in \Hi_q$ by $\varphi_{j,\ell} = \widehat \delta_{j + \ell q}$, that is:
\[
\varphi_{j,\ell}(\theta)
=
e^{-i \ell \theta} \vec e_j,
\]
where $\vec e_1, \ldots, \vec e_q$ denotes the standard basis of $\bbC^q$. It is easy to see that $\{\varphi_{j,\ell} : 1 \le j \le q, \; \ell \in \bbZ\}$ is an orthonormal basis of $\Hi_q$. Armed with this knowledge, all claims in the lemma are straightforward.
\end{proof}

For any operator $B$ on $\ell^2(\bbZ)$, we denote by $\widehat B$ its action on Fourier space, that is, $\widehat B = \mathcal F B \mathcal F^{-1}$. Using the definition of $\mathcal{F}$, we see that $\widehat H$ acts as a multiplication operator on $\Hi_q$. That is, for any $f \in \Hi_q$, one has
\[
\left( \widehat H f \right)(\theta)
=
H_\theta f(\theta) \text{ for Lebesgue almost every } \theta \in [0,2\pi).
\]
Thus, even though $H$ has no eigenvalues and no point spectrum, we still colloquially say that $\mathcal{F}$ ``diagonalizes'' $H$. If we denote the eigenvalues of $H_\theta$ as $\lambda_1(\theta)  \leq \cdots \leq \lambda_q(\theta)$, then,
\[
\sigma(H)
=
\bigcup_{q-j \text{ even}} [\lambda_j(\pi),\lambda_j(0)]
\cup
\bigcup_{q-j \text{ odd}} [\lambda_j(0), \lambda_j(\pi)].
\]
This point of view takes the problem of determining the spectrum and replaces a numerically unstable problem (polynomial root-finding) and replaces it with a very stable procedure (computing eigenvalues of Hermitian matrices). Concretely, one can completely determine the spectrum by diagonalizing $H_0$ and $H_\pi$. There are fast, stable numerical algorithms for computing eigenvalues of finite tridiagonal Hermitian matrices; however, $H_0$ and $H_\pi$ are not tridiagonal, but rather ``tridiagonal with corner entries'' which can lead to numerical bottlenecks with a na\"ive implementation leading to needlessly performing $O(q^3)$ floating-point operations. However, conjugating $H_0$ and $H_\pi$ by a breadth-first permutation matrix creates penta-diagonal matrices which can be reduced to tridiagonal form very efficiently (in $O(q^2)$ operations). See \cite{PEF2015} for more details on this point of view.

\subsection{The Density of States}

Define a sequence of measures by
\begin{equation} \label{eq:DOSfirstdef}
\int g \, dk_N
=
\frac{1}{N} \tr(P_N \, g(H) \, P_N^*),
\end{equation}
for Borel sets $B$, where $P_N$ denotes projection onto coordinates $\{0,1,\ldots,N-1\}$. We will prove that $dk_N$ has a weak limit, $dk$, as $N\to\infty$, which we will call $dk$ the \emph{density of states measure} (DOS) for $H$.  The accumulation function of the DOS is called the \emph{integrated density of states} (IDS) and is denoted by
\[
k(E)
=
\int \chi_{(-\infty,E]} \, dk.
\]

\begin{theorem} \label{t.periodic.ids}
The DOS of $H$ exists. That is, $dk = \lim dk_N$ exists, where the limit is taken with respect to the weak topology. Denoting $\Sigma = \sigma(H)$, the IDS of $H$ is differentiable away from the edges of open gaps of $\Sigma$ {\rm(}i.e.\ on $\bbR \setminus \partial \Sigma${\rm)}; more explicitly, if $V$ is $q$-periodic, we have
\begin{equation}  \label{e.periodic.ids}
\frac{dk}{dE}(E_0)
=
\frac{|D'(E_0)|}{\pi q \sqrt{4- D(E_0)^2} }
=
\frac{1}{\pi q} \left| \frac{d\theta}{dE}(E_0) \right|,
\end{equation}
whenever $D(E_0) \in (-2,2)$, where $\theta = \theta(E)$ is chosen continuously so that
\begin{equation}\label{eq:pertheta:def}
2 \cos(\theta) = D(E),
\quad
E \in \Sigma.
\end{equation}
If $E_0$ is a closed gap of $\Sigma$, then
\begin{equation} \label{e.periodic.ids.closedgap}
\frac{dk}{dE}(E_0)
=
\lim_{E \to E_0} \frac{|D'(E)|}{\pi q \sqrt{4- D(E)^2} };
\end{equation}
in particular, the limit on the right-hand side exists and is finite. Otherwise, when $E_0 \in \partial \Sigma$, then $E_0$ borders an open gap of $\Sigma$, and $dk/dE$ diverges at the rate $|E-E_0|^{-1/2}$ as $E \in \Sigma$ approaches $E_0$. Moreover, if $B=[\alpha,\beta]$ is any band of the spectrum,
\begin{equation} \label{eq:per:bandmeas}
\int_B \! dk
=
k(\beta)-k(\alpha)
=
\frac{1}{q}.
\end{equation}
\end{theorem}

\begin{proof}
For $m \in \bbZ_+$, consider $H_{m}^{\per}$, the restriction of $H$ to $[1,mq]$ with periodic boundary conditions, i.e.,
$$
H_{m}^{\per} =
\begin{bmatrix}
V(1) & 1    &        &        & 1 \\
1    & V(2) & 1      &        &   & \\
     & \ddots & \ddots & \ddots &   & \\
     &      & 1 & V(mq-1) & 1 & \\
1    &      &        & 1      & V(mq)
\end{bmatrix},
$$
and denote the corresponding eigenvalues by $E_{m,1}^{\per} \leq \cdots \leq E_{m,mq}^{\per}$ (notice that eigenvalues of $H_m^{\per}$ may occur with multiplicity two).

For $m \in \bbZ_+$, consider the probability measure
$$
dk_m^{\per}
=
\frac{1}{mq} \sum_{j=1}^{mq} \delta_{E_{m,j}^{\per}}
$$
defined by uniformly distributing point masses on the eigenvalues of $H_m^{\per}$. One can easily check that $E$ is an eigenvalue of $H_m^{\per}$ if and only if 1 is an eigenvalue of $[\Phi(E)]^m$, which holds if and only if the monodromy matrix $\Phi(E)$ has an $m$th root of unity as an eigenvalue.  Thus, $E$ is an eigenvalue of $H_m^{\per}$ if and only if
$$
D(E) = 2 \cos \left( \frac{2 \pi j}{m} \right)
$$
for some integer $ 0 \leq j \leq m/2 $.  Notice that any such eigenvalue with $0 < j < m/2$ is necessarily of multiplicity two, while the eigenvalues corresponding to $j=0$ and $j=m/2$ (for even $m$) can have multiplicity one or two.

Let $B_n$ denote the $n$th band of $ H $ and $D_n$ the restriction of $ D $ to $ B_n $ so that $D_n^{-1}$ is a continuous map $[-2,2] \to B_n$. Given a continuous compactly supported function $g$, the previous discussion implies
$$
\lim_{m\to\infty}
\int_{B_n} g(E) \, dk_m^\per(E) = \frac{1}{\pi q} \int_0^{\pi} g \left( D_n^{-1} ( 2 \cos(\theta) ) \right) \, d\theta.
$$
From the existence of the weak limit of $dk_m^\per$, one may deduce that $dk_N$ converges weakly and to the same limit. From this and the foregoing work, the first equality in \eqref{e.periodic.ids} follows by using the change of variables suggested in \eqref{eq:pertheta:def} and keeping track of the sign of $D'$; the second is a consequence of the definition of $\theta$. This also shows that $dk/dE$ diverges as the inverse square root at borders of open gaps, since $D'(E_0) \neq 0$ whenever $E_0$ borders an open gap of $\Sigma$ (which may be seen from \eqref{eq:discriminantCharPoly}). Now, if $E_0$ is a closed gap of $\Sigma$, then the limit in \eqref{e.periodic.ids.closedgap} exists since $D'$ vanishes linearly at $E_0$ and $4-D^2$ vanishes quadratically at $E_0$.
\end{proof}

\subsection{Ballistic Motion} \label{ss:ballistic}

Let us consider the position operator:
\[
[X\psi]_n
=
n \psi_n,
\]
which is self-adjoint on $D(X) = \set{\psi \in \ell^2(\bbZ) : X\psi \in \ell^2(\bbZ)}$. We are interested in the Heisenberg evolution when $H$ is periodic:
\[
X(t)
=
e^{itH} X e^{-itH},
\quad
t \in \bbR.
\]
One can check that $e^{itH}$ preserves $D(X)$ and that $X(t)$ is self-adjoint on $D(X)$ for all $t$.

\begin{theorem} \label{t:DLY}
If $H$ is periodic, the strong limit
\[
Q
=
\slim_{t\to\infty} t^{-1}X(t)
\]
exists and $\ker(Q) \cap D(X) = \{0\}$.
\end{theorem}

In fact, the conclusion of this theorem applies to any operator family with a Bloch-like decomposition. It was proved first for continuum operators by Asch and Knauf \cite{AscKna1998Nonlin}, and then for block Jacobi operators \cite{DLY2015} (which include discrete Schr\"odinger operators as a special case), CMV matrices, and quantum walks \cite{AVWW,DFO}.
\bigskip

We will sketch the main strokes of the proof and refer the reader to \cite{DLY2015} for details in the present setting. Formally (we are avoiding some niceties regarding domains, since $X$ is unbounded), differentiating $X(t)$ with respect to time, one obtains:
\[
\frac{dX}{dt}
=
i e^{itH}(HX - XH) e^{-itH}
=
A(t),
\]
where
\[
A :=
i[H,X]
=
i(HX - XH).
\]
Then, one can identify
\begin{equation} \label{eq:XtoA}
\frac{1}{t}(X(t) - X)
=
\frac{1}{t} \int_0^t A(s) \, ds.
\end{equation}
If $H$ and $A$ were finite matrices, the existence of the limit on the right-hand side of \eqref{eq:XtoA} would be trivial. To wit, if $H_0$ and $A_0$ are $m \times m$ matrices with $H_0$ Hermitian-symmetric, then, writing the spectral decomposition of $H_0$ as
\[
H_0
=
\sum_{\lambda \in \sigma(H_0)} \lambda P_\lambda,
\]
one can readily compute that
\begin{equation} \label{eq:diagPart}
\lim_{t\to\infty} \frac{1}{t} \int_0^t e^{isH_0} A_0 e^{-isH_0} \, ds
=
\sum_{\lambda \in \sigma(H_0)} P_\lambda A_0 P_\lambda,
\end{equation}
the ``diagonal part'' of $A_0$ with respect to $H_0$. More generally, this argument obtains for operators having pure point spectrum. However, as $H$ has purely absolutely continuous spectrum, one has to work a little harder. The main idea is to apply \eqref{eq:diagPart} on the Floquet fibers to obtain the existence of the limit defining $A$ and compute the diagonal part explicitly (fiberwise) to verify that $\ker(Q)$ is trivial.

Thus, the Bloch decomposition from Section~\ref{ss:bloch} is the key ingredient in the proof of Theorem~\ref{t:DLY}. Recall that we defined the direct integral space $\Hi_q$ by
\begin{align*}
\Hi_q
& :=
L^2\left([0,2\pi), \bbC^q; \frac{d\theta}{2\pi} \right) \\
& =
\set{f:[0,2\pi) \to \bbC^q : \int_0^{2\pi} \|f(\theta)\|_{\bbC^q}^2 \, \frac{d\theta}{2\pi} < \infty}.
\end{align*}

Let $\mathcal F$ be the Fourier transform defined by \eqref{eq:modqFourierDef}. It is not hard to see that $\mathcal F$ also diagonalizes the commutator $A$ into a multiplication operator. Specifically, we have $\left( \widehat A f \right)(\theta) = A_\theta f(\theta)$ for $f \in \mathcal H_q$ and Lebesgue a.e.\ $\theta \in [0,2\pi)$, where
\[
A_\theta
=
i \begin{bmatrix}
0 & 1 &&& -e^{-i\theta} \\
-1 & 0 & 1 && \\
& \ddots & \ddots & \ddots & \\
&&-1 & 0 & 1 \\
e^{i\theta} &&& -1 & 0
\end{bmatrix}.
\]
Now let $\Lambda_\theta \in \bbC^{q \times q}$ denote the diagonal matrix with diagonal entries $\langle e_k, \Lambda_\theta e_k \rangle = e^{ik \theta/q}$, i.e.,
$$
\Lambda_\theta
=
\begin{bmatrix}
e^{i\theta/q} &&&& \\
& e^{2i\theta/q} &&& \\
&& \ddots && \\
&&& e^{i(q-1)\theta/q} & \\
&&&& e^{i\theta}
\end{bmatrix}.
$$
Define
$$
\widetilde H_\theta
=
\Lambda_\theta^* H_\theta \Lambda_\theta,
\quad
\widetilde A_\theta
=
\Lambda_\theta^* A_\theta \Lambda_\theta.
$$
By a direct calculation, one can verify that
\begin{equation} \label{eq:DLYmagic}
\frac{\partial}{\partial \theta} \widetilde H_\theta
=
\frac{1}{q} \widetilde A_\theta.
\end{equation}
The key calculation now is to compute the limit of the Heisenberg evolution of the truncated operators.
\begin{lemma}
For each $\theta \in [0,2\pi)$, let $\lambda_1(\theta) \le \cdots \leq  \lambda_q(\theta)$ denote the eigenvalues of $H_\theta$ {\rm(}enumerated with multiplicities{\rm)}. Then, for every $\theta \in [0,2\pi) \setminus \{0,\pi\}$, we have
\begin{equation} \label{eq:DLY9}
\lim_{t \to \infty} \frac{1}{t} \int_0^t \! e^{isH_\theta} A_\theta e^{-isH_\theta} \, ds
=
q \sum_{j=1}^q \frac{\partial \lambda_j}{\partial \theta}(\theta) P_j(\theta),
\end{equation}
where $P_j(\theta)$ denotes projection onto the eigenspace corresponding to $\lambda_j(\theta)$.
\end{lemma}

\begin{proof}[Proof Sketch]
Since $H_\theta$ is a finite Hermitian-symmetric matrix, the limit on the left-hand side of \eqref{eq:DLY9} exists and is simply equal to the diagonal part of $A_\theta$ with respect to the eigenspaces of $H_\theta$. Since $H_\theta$ has simple spectrum for $\theta \neq 0,\pi$, it then suffices to show
\[
\langle v_j(\theta), A_\theta v_j(\theta) \rangle
=
q \frac{\partial \lambda_j}{\partial \theta}(\theta)
\]
for $1 \le j \le q$ and $\theta \neq 0,\pi$. This follows from \eqref{eq:DLYmagic}. To see this, let $v_j(\theta)$ denote a continuous choice of a unit vector from the $\lambda_j(\theta)$ eigenspace of $H_\theta$, put $\widetilde v_j(\theta) := \Lambda_\theta^* v_j(\theta$) and observe that
\begin{align*}
\frac{\partial \lambda_j}{\partial \theta}(\theta)
& =
\frac{\partial}{\partial \theta} \langle v_j(\theta), H_\theta v_j(\theta) \rangle \\
& =
\frac{\partial}{\partial \theta} \langle \widetilde v_j(\theta), \widetilde H_\theta \widetilde v_j(\theta) \rangle \\
& =
\frac{1}{q} \langle \widetilde v_j(\theta), \widetilde A_\theta \widetilde v_j(\theta) \rangle \\
& =
\frac{1}{q} \langle  v_j(\theta), A_\theta  v_j(\theta) \rangle
\end{align*}
for $\theta \neq 0,\pi$.
\end{proof}

\begin{proof}[Proof of Theorem~\ref{t:DLY}] Once one has \eqref{eq:DLY9}, we get the  conclusions of Theorem~\ref{t:DLY} with
$$
Q
=
 \mathcal F^{-1} \left( \int_{[0,2\pi)}^\oplus \! \sum_{j=1}^q  q\frac{\partial \lambda_j}{\partial \theta}(\theta) P_j(\theta) \, \frac{d\theta}{2\pi} \right) \mathcal F.
$$
In particular, the limit exists by \eqref{eq:DLY9} and dominated convergence. The kernel of $Q$ is trivial, since $\partial_\theta\lambda_j(\theta) \neq 0$ for $\theta \neq 0,\pi$ (which may be proved using \eqref{eq:discriminantCharPoly} and implicit differentiation).
\end{proof}

\section{Almost-Periodic Hulls}\label{s.almostperiodic}

\subsection{Almost-Periodic Sequences}

The shift acts on $\ell^\infty(\bbZ)$ via
\[
[SV]_n
=
V_{n+1}.
\]
Given $V$, the \emph{orbit} of $V$ is the set of its translates:
\[
\orb(V)
=
\set{S^k V : k \in \bbZ};
\]
the \emph{hull} of $V$ is the closure of the orbit in $\ell^\infty$:
\[
\hull(V)
=
\overline{\orb(V)}^{\;\ell^\infty(\bbZ)}.
\]
Naturally, it is easy to see that $V$ is periodic if and only if $\hull(V)$ is finite. We say that $V$ is \textit{Bochner almost-periodic} if $\hull(V)$ is compact (in the $\ell^\infty(\bbZ)$ topology). We say that $V$ is \textit{Bohr almost-periodic} if, for every $\varepsilon>0$, the $\varepsilon$-almost periods of $V$ are relatively dense in $\bbZ$. That is, $V$ is Bohr almost-periodic if for every $\varepsilon>0$, there is an $R>0$ so that for any $k \in \bbZ$, there exists $\ell \in \bbZ$ with $|k-\ell|\leq R$ and
\[
\|S^\ell V  - V\|_\infty
<
\varepsilon.
\]

\begin{prop}
An element $V \in \ell^\infty(\bbZ)$ is almost-periodic in the Bohr sense if and only if it is almost-periodic in the Bochner sense.
\end{prop}

Thus, going forward, we only speak of potentials being \emph{almost-periodic} and freely use both characterizations.

\begin{proof}[Proof Sketch]
If $V$ is Bohr-almost periodic, let $\varepsilon > 0$ and cover $\hull(V)$ by finitely many $\varepsilon$-balls centered at elements of $\orb(V)$, say
\[
\hull(V)
\subseteq
\bigcup_{j=1}^n B_\varepsilon(S^{m_j} V).
\]
Take $R = \max|m_j|$. Then, given $k$, choose $m_j$ with $S^k V \in B_\varepsilon(S^{m_j}V)$ to get $\|S^{k-m_j}V - V\|_\infty < \varepsilon$ and hence $V$ is Bochner almost-periodic.
\medskip

Conversely, if $V$ is Bochner almost-periodic, then one can cover $\hull(V)$ by finitely many $\varepsilon$-balls centered at elements of $\orb(V)$. Thus, $\hull(V)$ is complete and totally bounded, hence compact.
\end{proof}

The operation
\[
(S^j V)\oast (S^kV)
=
S^{j+k}V
\]
makes $\orb(V)$ into a group. For abelian groups, one usually uses ``+'' or ``*''; we have eschewed these symbols to avoid confusion with addition and convolution of functions.

In the event that $V$ is periodic of minimal period $q$, it is clear that $\orb(V) \cong \bbZ_q$. Otherwise, when $V$ is aperiodic, $\orb(V) \cong \bbZ$. We should note that this isomorphism is purely as groups without topology; in particular, the reader is invited to verify that $\orb(V)$ \emph{never} has the discrete topology when $V$ is almost-periodic and aperiodic.

\begin{prop} \label{p:aphullgroup}
The operation $\oast$ extends uniquely to a continuous operation $\oast: \hull(V)^2 \to \hull(V)$. With respect to this operation, $\hull(V)$ is a compact abelian topological group.
\end{prop}

\begin{proof}
Since $S$ is an isometry, we have
\begin{align}
\label{eq:hullops:cont}
\|S^{k+\ell}V - S^{k'+\ell'}V\|_\infty
& =
\|S^{k-k'}V - S^{\ell'-\ell}V\|_\infty \\
\nonumber
& \le
\|S^{k-k'}V - V\|_\infty + \|S^{\ell'-\ell}V - V\|_\infty \\
\nonumber
& =
\| S^k V - S^{k'}V\|_\infty + \|S^{\ell'}V - S^\ell V\|_\infty.
\end{align}
This calculation shows that $\oast$ is uniformly continuous on $\orb(V) \times \orb(V)$ and hence extends uniquely to a continuous operation
\[
\oast : \hull(V) \times \hull(V) \to \hull(V).
\]
The abelian group properties of $(\hull(V),\oast)$ follow from those of $(\orb(V),\oast)$ and continuity.
\end{proof}

Clearly, if $V$ is almost-periodic, then $\Omega= \hull(V)$ is also a \emph{monothetic} group, that is, $\Omega$ contains a dense cyclic subgroup (namely $\orb(V)$). In view of this observation and Proposition~\ref{p:aphullgroup}, we can also characterize almost-periodic $V \in \ell^\infty(\bbZ)$ as precisely those sequences of the form
\begin{equation}\label{e.lpdyn}
V(n)
=
f(n\alpha),
\end{equation}
where $G$ is a compact, monothetic group whose topology is generated by a translation-invariant metric; $\alpha \in G$ generates a dense cyclic subgroup; and $f:G\to\bbR$ is continuous. In light of this connection, the theory of compact topological groups naturally plays an important role in the analysis of Schr\"odinger operators with almost-periodic potentials.

\subsection{Limit-Periodic Hulls}

One can characterize exactly what compact abelian groups are hulls of limit-periodic sequences.

\begin{prop} \label{p:lphullchar}
Every limit-periodic $V$ is almost-periodic. Moreover, an almost-periodic $V$ is limit-periodic if and only if $\hull(V)$ is totally disconnected. Finally, an almost-periodic $V$ is aperiodic if and only if no point of $\hull(V)$ is isolated.
\end{prop}

\begin{proof}
Suppose $V$ is limit-periodic. For each $j \in \bbZ_+$, choose $V_j \in \ell^\infty(\bbZ)$ and $q_j \in \bbZ_+$ such that $S^{q_j} V_j = V_j$ and
\[
\lim_{j \to \infty} \|V - V_j\|_\infty = 0.
\]
As before, it suffices to prove that $\orb(V)$ is totally bounded (since this also implies that $\hull(V)$ is totally bounded). Given $\varepsilon > 0$, choose $j$ such that
$$
\|V - V_j\|_\infty < \varepsilon.
$$
One can then check that $\orb(V)$ is contained in the following finite union of $\varepsilon$-balls:
$$
\bigcup_{k = 0}^{q_j - 1} B(S^kV_j,\varepsilon).
$$
Thus, $\hull(V)$ is complete and totally bounded, hence compact, so $V$ is almost-periodic. Let us show that $\hull(V)$ is totally disconnected. Since the hull of $V$ is a topological group, it suffices to show that there are arbitrarily small neighborhoods of the identity (i.e.\ $V$) that are simultaneously open and closed. Therefore, given $\varepsilon > 0$ we will show that $B(V,\varepsilon) \cap \hull(V)$ contains a non-empty set that is both closed and open. Choose $j$ so that $\|V_j - V\|_\infty \leq \frac{\varepsilon}{2}$. Then,
$$
\hull^{q_j}(V)
:=
\overline{\{ S^{kq_j} V : k \in \bbZ \}}
$$
is a compact subgroup of $\hull(V)$ of index at most $q_j$. Clearly, $\hull^{q_j}(V)$ is closed, but it is also open since it is the complement of the union of no more than $q_j-1$ other closed cosets. Moreover, since $S^{q_j} V_j = V_j$ and $\|V_j - V\|_\infty \leq \frac{\varepsilon}{2}$, every element $W \in \hull^{q_j}(V)$ satisfies $ \| W - V_j \|_\infty \leq \varepsilon/2 $ and hence $\hull^{q_j}(V)$ is contained in the $\varepsilon$-ball centered at $V$. Consequently, $\hull(V)$ is totally disconnected.
\medskip

Conversely, let $V$ be almost-periodic and suppose its hull is totally disconnected. We have to show that $V$ is limit-periodic. Given $\varepsilon > 0$, we have to find $W \in \ell^\infty(\bbZ)$ and $p \in \bbZ_+$ with $S^p W = W$ and $\|W - V\|_\infty < \varepsilon$. By Proposition~\ref{p:tdhausd:zd}, we may choose a compact open neighborhood $N$ of $V$ in $\hull(V)$, small enough so that
\begin{equation}\label{e.avilalp1}
\|(W_1 \oast W_2) - W_1 \|_\infty
<
\varepsilon/2
\quad
\text{ for all } W_1 \in \hull(V), \; W_2 \in N.
\end{equation}
This is possible since $\oast$ is uniformly continuous, $V$ is the identity of $\hull(V)$ with respect to $\oast$, and $\hull(V)$ is totally disconnected. Since the sets $N$ and $\hull(V) \setminus N$ are compact and disjoint, there exists $\delta > 0$ so that $\| X - Y \| \geq \delta$ for all $X \in N$ and all $Y \in \hull(V) \setminus N$. By almost-periodicity of $V$, we can choose $p \ge 1$ so that
$$
\|S^p V - V\|_\infty < \delta,
$$
hence $S^p V \in N$ by our choice of $\delta$. But then we find inductively that $\{ S^{kp}V : k \in \bbZ \} \subseteq N$, by isometry of $S$ and the choice of $\delta$. Now consider the $p$-periodic $W$ that coincides with $V$ on $[0,p-1]$. Given $n \in \bbZ$, we write $n = r + \ell p$ with $\ell \in \bbZ$ and $0 \le r \le p-1$. Then, it follows from \eqref{e.avilalp1} that
\begin{align*}
|V(n) - W(n)|
& =
| V(r + \ell p) - V(r)| \\[1mm]
& =
\left| (S^r V \oast S^{\ell p}V) (0) - (S^rV)(0) \right| \\[1mm]
& <
\varepsilon/2,
\end{align*}
since $S^{\ell p}V \in N$. This shows that the $p$-periodic $W$ obeys $\|W - V\|_\infty \leq \varepsilon/2 < \varepsilon$, concluding the proof of the first two claims.

For the final claim, note first that $\hull(V)$ consists only of isolated points when $V$ is periodic. Conversely, if $V$ is aperiodic (and almost-periodic), let $\varepsilon_j$ be any sequence converging to zero and pick $\varepsilon_j$ almost-periods $0<q_1< q_2 < \cdots$. Then $S^{q_j}V \to V$ (in $\ell^\infty$), which implies that $V$ is not isolated. We leave it to the reader to confirm that this implies that no point of $\hull(V)$ is isolated.
\end{proof}

In light of Proposition~\ref{p:lphullchar}, the following definition is natural.

\begin{definition}
A \emph{Cantor group} is a compact, totally disconnected group that has no isolated points.
\end{definition}

\begin{example}
The cyclic group $\bbZ_q$ is not a Cantor group, nor is the circle group $\bbT = \bbR/\bbZ$. Given a prime $p$, the group $\bbJ_p$ of $p$-adic integers is a Cantor group. The group $\bbJ_p$ admits a minimal translation, namely $T\omega = \omega+1$. More generally, if $\mathfrak{q} = (q_1,q_2, \ldots)$ is a sequence of positive integers so that $q_j | q_{j+1}$ for every $j$, the inverse limit
\[
\bbJ_{\mathfrak{q}}
=
\varprojlim \bbZ_{q_j}
\]
is a Cantor group; (the map from $\bbZ_{q_{j+1}}$ to $\bbZ_{q_j}$ is the canonical projection). For more about the $p$-adic integers and inverse limits, see the appendix.
\end{example}

In the course of the proof of Proposition~\ref{p:lphullchar}, we encountered a device that is quite useful in the study of limit-periodic operators and hulls, in the guise of the closed subgroup generated by $S^q V$, which we denoted $\hull^q(V)$. Namely, a Cantor group will have many compact finite-index subgroups. These compact subgroups are useful, as they provide a means of producing precisely those periodic potentials that can be represented via continuous functions on the hull. To be more specific, fix a monothetic Cantor group $\Omega$ and let us define $\CP$ to be the subset of $f \in C(\Omega,\bbR)$ with the property that $f \circ T^q = f$ for some $q \in \bbZ_+$. Clearly, if $f$ satisfies such an identify, one has
\[
V_\omega(n+q)
=
f(T^{n+q}\omega)
=
f(T^n\omega)
=
V_\omega(n),
\]
so that the associated potentials are periodic.

\begin{theorem} \label{t:pdense}
$\CP$ is dense in $C(\Omega,\bbR)$.
\end{theorem}

\begin{proof}
Since $\CP$ is an algebra and contains all constant functions, it suffices (by Stone--Weierstrass) to verify that $\CP$ separates points of $\omega$. To that end, let $\omega \neq \omega'$ be given.

 By Theorem~\ref{p:tdgroup:zd}, the open subgroups of $\Omega$ comprise a neighborhood basis for the topology of $\Omega$ at the identity. In light of this, choose $\Omega_0 \subseteq \Omega$ to be an open subgroup with $\omega - \omega' \notin \Omega_0$. Define a function $g:\Omega/\Omega_0 \to \bbR$ such that $g(\omega + \Omega_0) \neq g(\omega'+\Omega_0)$ (which is possible precisely because $\omega$ and $\omega'$ are inequivalent mod $\Omega_0$). Then,
 \[
 f(\omega)
 =
 g(\omega + \Omega_0)
 \]
 yields a continuous\footnote{Since $\Omega_0$ is open, $\Omega/\Omega_0$ has the discrete topology, so continuity of $f$ is free.} function with $f(\omega) \neq f(\omega')$. Moreover, since $\Omega_0$ is open and $\Omega$ is compact, $\Omega_0$ is a finite-index subgroup; denoting $q = \mathrm{index}(\Omega_0)$, one necessarily has $q\alpha \in \Omega_0$, which implies
 \[
 f(T^q x)
 =
 f(q\alpha + x)
 =
 f(x)
 \]
for all $x \in \Omega$, whence $f \in \CP$. Thus, $\CP$ is dense in $C(\Omega,\bbR)$ by Stone--Weierstrass.
\end{proof}

\section{Ergodic Schr\"odinger Operators}\label{s.ergodic}

We have seen that limit-periodic Schr\"odinger operators are those having potentials of the form
\[
V(n)
=
V_\omega(n)
=
f(T^n\omega),
\]
where $\omega$ is an element of $\Omega$ a monothetic metrizable Cantor group, $T$ is a minimal translation of $\Omega$, and $f:\Omega \to \bbR$ is continuous. Thus, limit-periodic operators naturally fall into the category of operators with ergodic, dynamically defined potentials. The present subsection will describe some tools, techniques, and objects that this dynamical formalism enables us to use in the study of ergodic Schr\"odinger operators.

\begin{definition}
Let $(\Omega,\mathcal{B},\mu)$ be a probability measure space. That is, $\Omega$ is a nonempty set, $\mathcal{B}$ is a $\sigma$-algebra of subsets of $\Omega$, and $\mu$ is a (positive) probability measure defined on $\mathcal{B}$. A $\mu$\emph{-ergodic transformation} $T:\Omega \to \Omega$ is a measurable transformation which is $\mu$\emph{-preserving} in the sense that
\[
\mu(T^{-1}E) = \mu(E) \text{ for every } E \in \mathcal{B}
\]
and which has the property that $\mu(E) \in \{0,1\}$  whenever $E \subseteq \Omega$ satisfies $T^{-1}E = E$. We will also interchangeably say that $\mu$ is a $T$\emph{-ergodic measure}.

Suppose $T$ is invertible and that $T^{-1}$ is also measurable, and let $f:\Omega \to \bbR$ denote a bounded, measurable function. The associated family of \emph{ergodic Schr\"odinger operators} is defined by $H_\omega = \Delta + V_\omega$, where
\[
V_\omega(n)
=
f(T^n\omega), \quad \omega \in \Omega, \; n \in \bbZ.
\]
\end{definition}

In view of Proposition~\ref{p:lphullchar}, limit-periodic operators fall into this categorization by taking $\Omega = \hull(V)$, $T\omega = \omega \oast SV$, and $f(\omega) = \omega(0)$. The notation becomes somewhat redundant in that case, since one has $V_\omega = \omega$ for $\omega \in \hull(V)$, but we will continue to write $V_\omega$ to better match current notational conventions in the literature. In this setting, one can check that the normalized Haar measure is the unique $T$-ergodic measure on $\Omega$.

Throughout the remainder of the present section, we fix a measure space $(\Omega,\mathcal{B},\mu)$, an invertible $\mu$-ergodic transformation $T:\Omega\to\Omega$, a bounded measurable $f:\Omega \to \bbR$, and we let $\{H_\omega\}_{\omega \in \Omega}$ denote the associated family of ergodic Schr\"odinger operators. In this setting, one can leverage tools, techniques, and ideas from dynamical systems to prove results that hold $\mu$-almost surely. First, one can alternatively characterize ergodicity in the following manner: if $f:\Omega \to \bbR$ is measurable and $T$\emph{-invariant} in the sense that $f\circ T = f$, then $f$ is almost-surely constant; that is, there exists $c \in \bbR$ such that $f(\omega) = c$ for $\mu$-a.e.\ $\omega \in \Omega$. In the present formalism, if $S:\ell^2(\bbZ) \to \ell^2(\bbZ)$ denotes the left shift $\delta_n \mapsto \delta_{n-1}$, it is easy to verify that
\begin{equation} \label{eq:covariance}
H_{T\omega}
=
SH_\omega S^*,
\end{equation}
so $H_{T\omega}$ is unitarily equivalent to $H_\omega$ via the shift. By an induction followed by a limiting argument, it follows that \eqref{eq:covariance} holds for sufficiently nice functions of the operator, e.g.,
\[
g(H_{T\omega})
=
Sg(H_\omega)S^*,
\quad
g \in C(\bbR).
\]
Consequently, any spectral data of $H_\omega$ is $T$-invariant so one should expect that if one can prove suitable measurability statements, then any spectral data of $H_\omega$ to be almost-surely constant with respect to the ergodic measure $\mu$. As a sample result, this holds for the spectrum as a set as well as the spectral decomposition with respect to Lebesgue measure.

\begin{theorem}
There exist compact sets $\Sigma, \Sigma_\ac,\Sigma_\sc,\Sigma_\pp \subseteq \bbR$ with the property that
\begin{align*}
\sigma(H_\omega) & = \Sigma, \; \mu\text{-almost surely} \\
\sigma_\ac(H_\omega) & = \Sigma_\ac, \; \mu\text{-almost surely} \\
\sigma_\sc(H_\omega) & = \Sigma_\sc, \; \mu\text{-almost surely} \\
\sigma_\pp(H_\omega) & = \Sigma_\pp, \; \mu\text{-almost surely}.
\end{align*}
Moreover, one also has
\[
\sigma_{\mathrm{disc}}(H_\omega)
=
\emptyset,
\; \mu\text{-almost surely,}
\]
and, for any $E \in \bbR$,
\[
\mu\{\omega \in \Omega : E \text{ is an eigenvalue of } H_\omega\}
=
0.
\]
\end{theorem}

The previous theorem summarizes results of Kunz--Souillard \cite{KS80} and Pastur \cite{Pastur1980CMP}. Consult \cite[Chapter~9]{CFKS} for proofs. See also \cite{CarmonaLacroix1990,DF201X,PasturFigotin1992:ESO}. Imposing additional assumptions on the base dynamics $(\Omega,T)$ and the sampling function $f$ can allow one to draw stronger conclusions. For example, in the case where $\Omega$ is the hull of an almost periodic sequence and $T$ is the shift map, the dynamical system $(\Omega,T)$ is \emph{minimal} in the sense that $\{T^n \omega : n \in \bbZ\}$ is dense in $\Omega$ for all $\omega \in \Omega$ and the sampling function $f$ is continuous. Under the assumptions of minimality of $(\Omega,T)$ and continuity of $f$, one can use strong operator approximation to upgrade the almost-sure spectrum to a completely-sure spectrum. That is, one has
\[
\sigma(H_\omega)
=
\Sigma
\text{ for all }\omega \in \Omega.
\]
One can prove this by using very general results, e.g.\ \cite[Theorem~VIII-1.14]{Kato1980:PertTh} or \cite[Theorem~VIII.24]{ReedSimon1}. It is a deep result of Last and Simon that this also holds true for the absolutely continuous spectrum. That is, if $(\Omega,T)$ is minimal and $f$ is continuous, then $\sigma_\ac(H_\omega) = \Sigma_\ac$ for \emph{all} $\omega \in \Omega$ \cite{LastSim1999Invent}.

\subsection{The Lyapunov Exponent}

Motivated by \eqref{eq:diffEq2}, define
\[
A_z(\omega)
=
\begin{bmatrix} z - f(T\omega) & -1 \\
1 & 0
\end{bmatrix},
\quad
\omega \in \Omega, \; z \in \bbC.
\]
For $n \in \bbZ$, let
\[
A_z^n(\omega)
=
\begin{cases}
A_z(T^{n-1}\omega) A_z(T^{n-2}\omega) \cdots A_z(\omega) & n > 0,\\
I & n = 0,\\
A_z^{-n}(T^n\omega)^{-1} & n < 0.
\end{cases}
\]
Thus, if $u \in \bbC^\bbZ$ solves $H_\omega u = zu$ in the difference equation sense, one has
\[
\begin{bmatrix} u(n+1) \\ u(n) \end{bmatrix}
=
A_z^n(\omega)
\begin{bmatrix} u(1) \\ u(0) \end{bmatrix}
\]
for all $n \in \bbZ$.

\begin{definition}

The \emph{Lyapunov exponent} of the ergodic family is defined by
\[
L(z)
=
\lim_{n\to\infty} \frac{1}{n} \int \log\|A_z^n(\omega) \| \, d\mu(\omega),
\quad
z \in \bbC.
\]
\end{definition}

By Kingman's subadditive ergodic theorem \cite{Kingman1973}, for each fixed $z \in \bbC$,
\begin{equation} \label{eq:kingmanConseq}
L(z)
=
\lim_{n\to\infty} \frac{1}{n} \log \| A_z^n(\omega)\|
\end{equation}
for $\mu$-a.e.\ $\omega \in \Omega$. In general, one cannot reverse the quantifiers. That is to say, it is not the case in general that there exists a $z$-independent full-measure set of $\Omega_{\mathrm{L}} \subseteq \Omega$ with the property that \eqref{eq:kingmanConseq} holds simultaneously for all $z \in \bbC$ and every $\omega \in \Omega_{\mathrm{L}}$. In general, one can partially reverse the quantifiers by applying Fubini's theorem on a suitable product space. For example, one can say that, for $\mu$-a.e.\ $\omega$, one has \eqref{eq:kingmanConseq} for Lebesgue a.e.\ $E \in \bbR$.
\bigskip

\subsection{The Density of States}

Next, we consider the density of states (DOS) in the ergodic setting. For each $\omega \in \Omega$ and each $N \in \bbZ_+$, we may define $dk_{\omega,N}$ as in \eqref{eq:DOSfirstdef}, that is,
\[
\int g(E) \, dk_{\omega,N}
=
\frac{1}{N} \tr(P_N \, g(H_\omega) \, P_N^*),
\]
for continuous $g$, where $P_N$ is projection onto coordinates in $[0,N)$. In this setting, the weak limit of $dk_{\omega,N}$ (i.e.\ the DOS) exists $\mu$-almost surely and is deterministic.

\begin{theorem} \label{t:DOSexists}
There is a deterministic probability measure $dk$ and full-measure subset $\Omega_{\mathrm{DOS}} \subseteq \Omega$ with the property that $dk_{\omega,N}$ converges weakly to $dk$ for all $\omega \in \Omega_{\mathrm{DOS}}$. Moreover,
\begin{equation} \label{eq:ergDOSdef}
\int g \, dk
=
\int_\Omega \langle \delta_0, g(H_\omega) \delta_0 \rangle \, d\mu(\omega)
\end{equation}
for $g \in C(\bbR)$.
\end{theorem}

\begin{proof}
Let us consider a continuous $g:\bbR \to \bbR$ and define $\widetilde{g}:\Omega \to \bbR$ by
\[
\widetilde{g}(\omega)
=
\langle \delta_0, g(H_\omega) \delta_0\rangle.
\]
By definition of $dk_{\omega,N}$, we then observe that
\begin{align*}
\int g \, dk_{\omega,N}
& =
\frac{1}{N} \tr(P_N \, g(H_\omega) \, P_N^*) \\
& =
\frac{1}{N} \sum_{n=0}^{N-1} \langle \delta_n, g(H_\omega) \delta_n \rangle.
\end{align*}
Applying \eqref{eq:covariance}, we obtain
\begin{equation} \label{eq:DOSlim}
\int g \, dk_{\omega,N}
=
\frac{1}{N} \sum_{n=0}^{N-1} \widetilde{g}(T^n\omega).
\end{equation}

By Birkhoff's ergodic theorem, we know that the right-hand side converges to the integral of $\widetilde g$ $\mu$-almost surely, and hence \eqref{eq:ergDOSdef} holds by definition of $\widetilde{g}$. This provides a $g$-dependent full measure set $\Omega_g$ on which \eqref{eq:ergDOSdef} holds; to obtain $\Omega_{\mathrm{DOS}}$, note that boundedness of the sampling function $f$ implies that there is a uniform compact set $K$ that contains the support of $dk_{\omega,N}$ for all $\omega$ and $N$; then, choose a countable dense set $\{g_n : n \ge 1\} \subseteq C(K,\bbR)$, take
\[
\Omega_{\mathrm{DOS}}
=
\bigcap_n \Omega_{g_n},
\]
and apply an $\varepsilon/3$ argument.
\end{proof}

For almost-periodic (and in particular limit-periodic) operators, the situation is even better. Since normalized Haar measure is the unique invariant measure on the hull, an argument using \emph{unique} ergodicity implies that $dk_{\omega,N}$ converges weakly to $dk$ for \emph{all} $\omega$ in the hull. Specifically, unique ergodicity implies that the limit of the quantity on the right hand side of \eqref{eq:DOSlim} exists for all $\omega$, not just $\mu$-a.e.\ $\omega \in \Omega$. Consequently, if $V$ is limit-periodic, then $dk_{V,N}$ converges weakly to $dk$, and we may speak of ``the density of states of $V$'' without worry.

We will be concerned with the regularity of the DOS and IDS.  In full generality, the IDS is continuous \cite{DelSou1984CMP}.

\begin{theorem} \label{t:DelSou}
For any ergodic family, the integrated density of states is a continuous function of $E$. Equivalently, the DOS is an atomless measure.
\end{theorem}

\begin{proof}
Fix $E_0 \in \bbR$ and suppose $g_n$ are continuous compactly supported functions with $g_n(E_0) = 1$ and $g_n(E) \downarrow 0$ for $E \not= E_0$. Then, by dominated convergence,
\begin{equation}\label{fnlimpos}
\int g_n \, dk \to \int \chi_{\{E_0\}} \, dk.
\end{equation}
Denoting $A_\omega = \chi_{\{E_0\}}(H_\omega)$, we then also have $\langle \delta_0 , g_n(H_\omega) \delta_0 \rangle \to \langle \delta_0 , A_\omega \delta_0 \rangle$ for all $\omega$, and hence, again by dominated convergence,
\begin{align*}
\int g_n \, dk
& =
\int  \langle \delta_0 , g_n(H_\omega) \delta_0 \rangle  \, d\mu(\omega) \\
& \to
\int  \langle \delta_0 ,
A_\omega \delta_0 \rangle  \, d\mu(\omega).
\end{align*}
For $\omega$'s from a set of full $\mu$-measure,\footnote{To get this full-measure set, apply the argument from Theorem~\ref{t:DOSexists} to the function $g = \chi_{\{E_0\}}$.}
$$
\int  \langle \delta_0 , A_\omega \delta_0 \rangle  \, d\mu(\omega)
 =
 \lim_{N \to \infty} \frac{1}{N} \, \tr (P_N \, A_\omega \, P_N^*).
$$
Since $\tr(P_N A_\omega P_N^*)$ is bounded by $1$, the limit is zero. Thus, $dk(\{E_0\}) = 0$, and $dk$ is a continuous measure. Since $k$ is the accumulation function of $dk$, continuity of $k$ follows immediately.
\end{proof}

\begin{theorem}\label{t.as83}
The almost sure spectrum is given by the points of increase of $k$, that is, $\supp(dk) = \Sigma$.
\end{theorem}

\begin{proof}
If $E_0 \not\in \Sigma$, there is an open interval $I$ containing $E_0$ with $I \cap \Sigma = \emptyset$. We then have $\chi_I(H_\omega) = 0$ for a.e.\ $\omega$ and hence
$$
\int \chi_I \, dk
=
\int_\Omega  \langle \delta_0 , \chi_I(H_\omega) \delta_0 \rangle  \, d\mu(\omega)
=
0.
$$
Thus, $E_0 \not\in \supp (dk)$.

Conversely, if $E_0 \not\in \supp (dk)$, there is an interval $I$ containing $E_0$ such that $I \cap \supp(dk) = \emptyset$. Then,
\begin{align*}
\int \chi_I \, dk
& =
\int  \langle \delta_0 , \chi_I(H_\omega) \delta_0 \rangle  \, d\mu(\omega) \\
& =
\int  \langle \delta_0 , \chi_I(H_{T^n \omega}) \delta_0 \rangle  \, d\mu(\omega)\\
& =
\int  \langle \delta_n , \chi_I(H_\omega) \delta_n \rangle  \, d\mu(\omega).
\end{align*}
By positivity, we obtain
\[
\dim \Ran \, \chi_I(H_\omega)
=
\tr \,\chi_I(H_\omega)
=
0
\]
for $\mu$-a.e.\ $\omega \in \Omega$ and hence $E_0 \not\in \Sigma$.
\end{proof}

The density of states and the Lyapunov exponent are related via the \emph{Thouless formula}:
\[
L(z)
=
\int \log|E-z| \, dk(E),
\quad
z \in \bbC.
\]
This was discovered on an intuitive basis by Thouless in the early 1970s \cite{Thouless1972JPhysC}, with the first rigorous proof due to Avron--Simon \cite{AvrSim1983DMJ}.

\section{Absolutely Continuous Spectrum}\label{s.ac}

In this section, we shall consider the limit-periodic potentials in the perturbative regime.  Concretely, we will consider potentials $V$ with $q_n$-periodic approximants such that
\begin{equation} \label{eq:ptdef}
\lim_{n \to \infty} e^{bq_{n+1}} \|V_n- V\|_\infty
=
0
\text{ for every } b>0.
\end{equation}
In this case, we will write $V \in \PT$, after the contributions of \cite{PT1, PT2}; the inverse spectral problem for Jacobi matrices satisfying a condition like \eqref{eq:ptdef} was studied by Egorova \cite{egorova}. See also~\cite{Chul81,Chul84,MC84}. The main point of the present section is that potentials in $\PT$ behave extremely similarly to periodic potentials.

We consider first the spectrum as a set. In the case when $V$ is $q$-periodic, we have seen that the spectrum consists of $q$ nondegenerate closed intervals, each of which is given as the closure of a connected component of $D^{-1}((-2,2))$, where $D$ is the Floquet discriminant. The intervals may touch at the endpoints but do not overlap otherwise. It should not  be surprising that this behavior does not persist for genuine aperiodic limit-periodic models. The number of spectral bands grows with the period, and the spectrum must lie within the interval
\[
I = [-2-\|V\|_\infty, 2+\|V\|_\infty],
\]
and so we cannot expect $\sigma(H)$ to consist of nontrivial closed intervals when $V$ is aperiodic. Indeed, one can get Cantor sets for $V \in \PT$:

\begin{theorem}
For a dense set $\mathcal{C} \subseteq \PT$, $\sigma(H_V)$ is a Cantor set for all $V \in \mathcal{C}$.
\end{theorem}

However, for $V \in \PT$, the spectrum of $H$ is thick in a precise sense that is nice from the point of view of inverse spectral theory. Concretely, one says that a set $\Sigma \subseteq \bbR$ is $\tau$\emph{-homogeneous} if there exists $\delta_0 > 0$ with the property that
\[
\Leb(\Sigma \cap (x - \delta, x+\delta))
\geq
\delta \tau
\]
for all $0 < \delta \leq \delta_0$ and all $x \in \Sigma$. For instance, it is easy to see that $\sigma(H)$ is 1-homogenous if $V$ is periodic. The property of homogeneity persists for all $V \in \PT$, which is due to Fillman and Lukic \cite{FL2017}.

\begin{theorem}
If $V \in \PT$, then $\sigma(H_V)$ is $\tau$-homogeneous for every $\tau<1$.
\end{theorem}

\begin{proof}
We will describe how to prove homogeneity with $\tau = 1/2$; modifying the proof for $1/2 < \tau < 1$ involves fiddling with constants. If $V \in \PT$, let $V_j \to V$ be such that \eqref{eq:ptdef} holds, and abbreviate $H_j = \Delta + V_j$. There is a constant that only depends on $\|V\|_\infty$ with the property that each band of $\sigma(H_j)$ has length at least $K^{-q_j}$.
Using the PT condition, we may remove finitely many terms of the sequence $\{V_n\}_{n=1}^\infty$ and renumber to ensure that
\begin{equation} \label{eq:smalltail}
\sum_{n=1}^\infty
K^{q_{n+1}} \|V_n - V_{n+1}\|_\infty
<
\frac{1}{10}.
\end{equation}
Put $\delta_0 = K^{-q_1}$. We will prove the following estimate:
\begin{equation} \label{eq:sbs:homog:est}
|B_\delta(x) \cap \Sigma_N|
\geq
\delta/2
\text{ for all }
x \in \Sigma_N
\text{ and every }
0 < \delta \leq \delta_0
\end{equation}
for all $N \in \bbZ_+$. To that end, fix $N \in \bbZ_+$, $x \in \Sigma_N$, and $0 < \delta \leq \delta_0$. If $\delta \leq K^{-q_N}$, \eqref{eq:sbs:homog:est} is an obvious consequence of our choice of $K$, since $\delta$ is less than the length of the band of $\Sigma_N$ which contains $x$ in this case.  Otherwise, $\delta > K^{-q_N}$, and there is a unique integer $n$ with $1 \leq n \leq N-1$ such that
\begin{equation}\label{deltachoice}
K^{-q_{n+1}}
<
\delta
\leq
K^{-q_n}.
\end{equation}
This integer $n$ is relevant, as it determines the periodic approximant corresponding to the length scale $\delta$. More precisely, by our choice of $K$, any band of $\Sigma_n$ has length at least $\delta$. On the other hand, by Lemma~\ref{l:spec:pert}, there exists $x_0 \in \Sigma_n$ with
\begin{equation}
|x-x_0|
\leq
\sum_{\ell = n}^{N - 1}
\|V_\ell - V_{\ell+1}\|_\infty.
\end{equation}
 Using \eqref{deltachoice}, we deduce
\begin{align}
\nonumber
\lvert x-x_0 \rvert
& \leq
\sum_{\ell = n}^{N - 1}
\|V_\ell - V_{\ell+1}\|_\infty \\
\nonumber
& < \delta K^{q_{n+1}}  \sum_{\ell= n}^{N - 1}  \|V_\ell - V_{\ell+1}\|_\infty \\
\nonumber
& < \delta  \sum_{\ell= n}^{N - 1} K^{q_{\ell+1}}  \|V_\ell - V_{\ell+1}\|_\infty \\
\label{eq:expsum2}
& < \frac{\delta}{10}.
\end{align}
Thus, there exists an interval $I_0$ with $x_0 \in I_0 \subseteq B_{\delta}(x) \cap \Sigma_n$ such that
$$
\Leb(I_0)
=
\delta - \frac{\delta}{10}
=
\frac{9 \delta}{10}.
$$
By standard measure theory, we have
$$
\Leb(B_\delta(x) \cap \Sigma_N)
\geq
\Leb(I_0 \cap \Sigma_n) - \sum_{\ell = n}^{N - 1}\Leb(I_0 \cap (\Sigma_\ell \setminus \Sigma_{\ell + 1})).
$$
Our choice of $K$ implies that the interval $I_0$ completely contains at most $\delta K^{q_{\ell + 1}}$ bands of $\Sigma_{\ell + 1}$ for each $\ell \geq n$. Consequently, perturbation theory (cf.\ Lemma~\ref{l:spec:pert}) yields
\begin{align*}
\Leb(I_0 \cap (\Sigma_\ell \setminus \Sigma_{\ell+1}))
& \leq
2 (\delta K^{q_{\ell+1}} + 1)
\| V_\ell - V_{\ell + 1} \|_\infty \\
& \leq
4\delta K^{q_{\ell+1}} \| V_\ell - V_{\ell + 1} \|_\infty.
\end{align*}
Notice that the extra term in the parentheses on the first line is needed to account for possible boundary effects. Summing this over $\ell$ and estimating the result with \eqref{eq:smalltail}, we obtain
\[
\sum_{\ell = n}^{N - 1} \Leb(I_0 \cap (\Sigma_\ell \setminus \Sigma_{\ell + 1} ))
\leq
\sum_{\ell=n}^{N-1} 4 \delta K^{q_{\ell+1}} \| V_\ell - V_{\ell + 1} \|_\infty
< \frac{2 \delta}{5}.
\]
Putting all of this together, we have
\[
|B_\delta(x) \cap \Sigma_N|
\geq
|I_0 \cap \Sigma_n| - \sum_{\ell = n}^{N-1}|I_0 \cap(\Sigma_\ell \setminus \Sigma_{\ell + 1})|
>
\frac{9 \delta}{10}
- \frac{2 \delta}{5}
=
\frac{\delta}{2}.
\]
This proves \eqref{eq:sbs:homog:est} for arbitrary $N \in \bbZ_+$. Since Lebesgue measure is upper semicontinuous with respect to the topology induced by the Hausdorff metric, we obtain
$$
|B_\delta(x) \cap \Sigma|
\geq
\delta/2
\text{ for all } x \in \Sigma, \text{ and } 0 < \delta \leq \delta_0,
$$
so $\Sigma$ is homogeneous, as promised.
\end{proof}

The motivation to study homogeneity of the spectrum of a Schr\"odinger operator (or Jacobi matrix) arises from inverse spectral theory. Loosely speaking, if $\Sigma \subseteq \bbR$ is homogeneous, then one has a very nice description of the set of Jacobi matrices with spectrum $\Sigma$ and which are reflectionless thereupon, due to M.\ Sodin and Yuditskii; all such operators are almost-periodic, and they comprise a torus whose dimension coincides with the number of gaps in $\Sigma$ \cite{SY97}. On homogeneous sets, one has a potent generalization of a theorem of Kotani for ergodic Schr\"odinger operators. Concretely, let $\{H_\omega\}$ be an ergodic family with Lyapunov exponent $L$. Kotani showed that if $L$ vanishes in an interval, then (almost surely with respect to the underlying ergodic measure) the spectrum of $H_\omega$ is purely absolutely continuous thereon (compare \cite[Theorem~4.2]{Kotani1984}). By a result of Poltoratski and Remling \cite{poltrem09}, one may replace ``interval'' by ``homogeneous set'' and deduce the pure absolute continuity of the spectrum (in fact, one can assume a ``weak'' version of homogeneity in which some constants are allowed to vary over the spectrum). There are also results for reflectionless continuum Schr\"odinger operators \cite{GY99, SY95} and CMV matrices \cite{GZ09}.

\begin{theorem}
If $V \in \PT$, then $H$ has purely a.c.\ spectrum.
\end{theorem}

\begin{proof}[Proof Sketch]
Let $\Omega = \hull(V)$, denote by $\Sigma$ the common spectrum of $H_\omega$ for all $\omega \in \Omega$, and let $\Sigma^{(q)}_\omega$ denote the spectrum of the $q$-periodic operator obtained by repeating the string $V_\omega(1), \ldots, V_\omega(q)$. From the definition of $\PT$ and perturbation theory, one can see that
\begin{equation} \label{eq:specConv}
\lim_{n\to\infty} \Leb\left(\Sigma \setminus \Sigma_\omega^{(q_n)}\right)
=
0
\end{equation}
for all $\omega$. Defining $\mathcal{Z} = \{z \in \bbC : L(z) = 0\}$ where $L$ denotes the Lyapunov exponent, one can use \eqref{eq:specConv} and a result of Last to show that $\Leb(\Sigma\setminus \mathcal{Z}) = 0$ \cite{Last1993CMP}. From this, Remling's Theorem implies that $H$ is reflectionless on $\Sigma$ \cite{REM11}. Consequently, the spectral type is purely absolutely continuous by homogeneity and Poltoratski--Remling \cite{poltrem09}.
\end{proof}

At the level of quantum dynamics, the potentials in $\PT$ also exhibit strong ballistic motion in the sense of Asch--Knauf, which is due to Fillman \cite{F2017CMP}.

\begin{theorem} \label{t:PTball}
If $V \in \PT$, then the Schr\"odinger group generated by $H_V$ exhibits strong ballistic motion in the sense that
\[
Q
=
\slim_{t\to\infty} t^{-1} X(t)
\]
exists and $\ker(Q) \cap D(X) = \{0\}$.
\end{theorem}

The key ingredient in the proof of Theorem~\ref{t:PTball} is to get precise quantitative estimates on the rate at which the strong convergence in Theorem~\ref{t:DLY} occurs. In a na\"ive sense, this looks easy. Thinking about how the convergence was originally proved, one can see easily that the convergence rate is roughly proportional to $t^{-1}$. However, the constant of proportionality depends on the length of the smallest gap, and thus the na\"ive estimates completely break when a gap degenerates; moreover, it is difficult to divine the length of the smallest gap from coefficient data alone; finally, even in the generic (``all gaps open'') scenario, since we are interested in the aperiodic case, one absolutely needs estimates that are independent of the period and the gap size. Consequently, one needs a more robust estimate that averages out resonances. One also needs to be sure that the estimate one obtains must have constants that do not depend on the period (since one would eventually need to send the period to $\infty$) or at least some quantitative control on the constants as functions of the period.

\begin{theorem}
For every $q$-periodic Schr\"odinger operator $H$, and every $\varphi \in D(X)$, one has
\begin{equation} \label{eq:ak:X:conv}
\left\|  Q \varphi  - \frac{1}{t} X(t) \varphi \right\|
\leq
 t^{-1} \|X \varphi \|  +  C_1^q \|\varphi\|_1 t^{-1/5},
\end{equation}
where $C_1$ denotes a constant that depends solely on $\|V\|_\infty$ and $\| \varphi \|_1$ denotes the $\ell^1$ norm of $\varphi$.
\end{theorem}

\section{Singular Continuous Spectrum}\label{s.sc}

Fix a monothetic Cantor group $\Omega$ and an $\alpha \in \Omega$ so that $\{ n \alpha : n \in \bbZ \}$ is dense in $\Omega$. If $f \in C(\Omega,\bbR)$ and $\omega \in \Omega$, then the potential given by
\begin{equation}\label{e.lpdynomega}
V_{\omega,f}(n) = f(\omega + n\alpha), \quad n \in \bbZ
\end{equation}
is limit-periodic. The associated Schr\"odinger operator is denoted by $H_{\omega,f}$ and the spectrum of $H_{\omega,f}$, which is independent of $\omega$ by minimality, is denoted by $\Sigma(f)$.

\begin{theorem} \label{t.lp.scspec.generic}
There exists a dense $G_{\delta}$ set $\mathcal{S} \subseteq C(\Omega,\bbR)$ such that for every $f \in \mathcal{S}$ and every $\omega \in \Omega$, the spectrum of $H_{\omega,f}$ is a Cantor set of zero Lebesgue measure and $H_{\omega,f}$ has purely singular continuous spectrum.
\end{theorem}

This theorem follows from two separate observations that lead to generic spectral results, the combination of which is the conclusion in Theorem~\ref{t.lp.scspec.generic}. The first observation is that one can carry out periodic approximation with control on the measure of the spectrum of the approximant. This leads to generic zero-measure spectrum, which as a byproduct also precludes any absolutely continuous spectral measures. The second observation shows that quantitative approximation with periodic potentials allows one to generically conclude that the limit-operators share an important property with the approximants: the absence of square-summable (and in fact decaying) solutions, which in turn shows that the spectral measures will also have no atoms. The next two subsections explain these two parts in more detail.

\subsection{Zero-Measure Spectrum}

\begin{theorem} \label{t.lp.cantorspec.generic}
There exists a dense $G_{\delta}$ subset $\mathcal{C} \subseteq C(\Omega,\bbR)$ so that $\Sigma(f)$ is a Cantor set of zero Lebesgue measure for all $f \in \mathcal C$.
\end{theorem}

\begin{remark}
In fact, one may pass from the generic setting to the dense setting and obtain even finer control on the size of the spectrum. Namely, there is a dense family whose spectra are of Hausdorff dimension zero. Both Theorem~\ref{t.lp.cantorspec.generic} and the strengthening described in this remark are due to Avila \cite{A09}.
\end{remark}

The proof of Theorem~\ref{t.lp.cantorspec.generic} is based on two key perturbative lemmas. Recall that the \emph{Hausdorff distance} between two compact nonempty sets $F, K \subset \bbR$ is defined by
\[
d_\Hd(F,K)
=
\inf\set{ \varepsilon > 0 : F \subset B_\varepsilon(K) \text{ and } K \subset B_\varepsilon(F)},
\]
where $B_\varepsilon(S)$ denotes the $\varepsilon$-neighborhood of the set $S$. Then the following estimate is not difficult to prove (exercise).

\begin{lemma} \label{l:spec:pert}
If $A$ and $B$ are bounded self-adjoint operators on a Hilbert space $\mathcal{H}$, then
\[
d_\Hd(\sigma(A),\sigma(B))
\le
\|A-B\|.
\]
\end{lemma}

The second key lemma shows that any sampling function may be uniformly approximated by periodic sampling functions whose spectra are exponentially small (with respect to the period). To formulate the lemma, we need to precisely say what periods a function in $\CP$ may have. To that end, fix a sequence $\Omega = \Omega_0 \supset \Omega_1 \supset \Omega_2 \supset \cdots$ of open subgroups of $\Omega$ with
\[
\bigcap_{n=1}^\infty \Omega_n
=
\{0\},
\]
where $0$ denotes the identity element of $\Omega$. Such a sequence exists by Proposition~\ref{p.small.open.subgroups} and metrizability of $\Omega$. Each $\Omega_n$ has finite index $q_n$ by compactness, and one clearly has $q_n \to \infty$ as $n\to\infty$. One can use $\Omega_n$ to define periodic sampling functions of period $q_n$ by the construction in the proof of Theorem~\ref{t:pdense}. That is, if $\tilde{f}:\Omega/\Omega_n \to \bbR$ is any function, then
\[
f(\omega)
=
\tilde{f}(\omega+\Omega_n)
\]
defines a $q_n$-periodic sampling function in $\CP$.

\begin{lemma} \label{l:lp:smallspec}
Fix $f \in C(\Omega,\bbR)$. For all $\varepsilon > 0$, there exist $c = c(f,\varepsilon) > 0$ and $n_0 = n_0(f,\varepsilon) \in \bbZ_+$ such that the following holds true. For all $n \geq n_0$, there exists $g = g_n \in C(\Omega,\bbR)$ of period $q_n$ such that
\[
\Leb(\Sigma(g))
\leq
e^{- c q_n}
\]
and $\|f - g\|_\infty < \varepsilon$.
\end{lemma}

\begin{proof}[Proof Sketch]
From Theorem~\ref{t:pdense}, $\mathcal P$ is dense in $C(\Omega,\bbR)$, so it suffices to prove the lemma for periodic sampling functions. Let $f \in C(\Omega,\bbR)$ be a given $q$-periodic sampling function. The construction works by first perturbing $f$ to produce a family of sampling functions which are very close to $f$, and whose resolvent sets cover the line. Thus, for every $E \in \bbR$, one of these new potentials will have $L(E) > 0$. We then form a new potential by concatenating these finite families over long blocks. Positive exponents over sub-blocks enable us to produce growth of transfer matrices, and one can then parlay growth of transfer matrices into upper bounds on band lengths.
\end{proof}

\begin{proof}[Proof of Theorem~\ref{t.lp.cantorspec.generic}]
For each $\delta > 0$, define
\[
U_{\delta}
=
\left\{ f \in C(\Omega,\bbR) : \Leb(\Sigma(f)) < \delta  \right\}.
\]
To prove the theorem, we will show that $U_{\delta}$ is open and dense for all $\delta > 0$. To that end, suppose given $f \in U_{\delta}$. In essence, small perturbations of $f$ remain in $U_{\delta}$ by continuity of the spectrum. The details follow.

Since $\Sigma(f)$ is a compact set with $\Leb(\Sigma(f)) < \delta$, we may choose finitely many open intervals $J_1,\ldots,J_n$ such that
\[
\Sigma(f)
\subseteq
\bigcup_{k=1}^n J_k,
\qquad
\sum_{k=1}^n \Leb(J_k) < \delta.
\]
Next, choose $\varepsilon > 0$ small enough that
\[
B_\varepsilon(\Sigma(f))
\subseteq
\bigcup_{k=1}^n J_k.
\]
Now, if $\|f - f'\|_\infty < \varepsilon$, then Lemma~\ref{l:spec:pert} implies
\[
\Sigma(f')
\subseteq
B_\varepsilon(\Sigma(f))
\subseteq
\bigcup_{k=1}^n J_k.
\]
Consequently, $\Leb(\Sigma(f')) < \delta$. Thus, $U_{\delta}$ is open.

\bigskip

On the other hand, Lemma~\ref{l:lp:smallspec} clearly implies that $U_{\delta}$ is dense for all $\delta > 0$. Consequently,
\[
\mathcal C
=
\bigcap_{k=1}^\infty U_{1/k}
\]
is a dense $G_\delta$ in $C(\Omega,\bbR)$ by the Baire Category Theorem; clearly $\Leb(\Sigma(f)) = 0$ for every $f \in \mathcal C$.
\end{proof}

\subsection{Continuous Spectrum}

In order to prove that spectral measures are continuous, one has to exclude eigenvalues. In other words, one needs to show that for any $E$, the difference equation \eqref{eq:diffEq1} does not admit any square-summable solutions $u$.

A classical sufficient condition is due to Gordon. This condition is applicable whenever one has very good approximation of a given potential $V$ by periodic potentials over a (rather small) fixed number of periods around the origin. Clearly, for limit-periodic $V$, the number of periods for which one has the desired rate of approximation may be arbitrary, so all one needs to take care of is the size of the error relative to the size of the period. It turns out that the absence of eigenvalues can be established in this way for a generic set of sampling functions.

Let us begin by recalling the abstract Gordon criterion. Due to the reformulation \eqref{eq:diffEq2} of \eqref{eq:diffEq1}, a sequence $u \in \bbC^\bbZ$ solves \eqref{eq:diffEq1} if and only if the sequence $\mathbf{U} :\bbZ \to \bbC^2$ defined by
\begin{equation} \label{d:2by2:u2}
\mathbf{U}(n)
=
\begin{bmatrix}
u(n+1) \\ u(n)
\end{bmatrix},
\quad n \in \bbZ
\end{equation}
solves
\begin{equation}\label{eq:diffEq}
\mathbf{U}(n) = M_E(n) \mathbf{U}(0),
\end{equation}
where
$$
M_E(n) =
\begin{cases}
A_E(n) A_E(n-1) \cdots A_E(1) & n > 0,\\
I & n = 0,\\
A_E(n+1)^{-1} A_E(n+2)^{-1} \cdots A_E(0)^{-1} & n < 0
\end{cases}
$$
and
$$
A_E(m)
=
\begin{bmatrix} E - V(m) & -1 \\
1 & 0
\end{bmatrix}.
$$

\begin{lemma}\label{l.gordon}
Suppose $V : \bbZ \to \bbR$ obeys $V(n+q) = V(n)$ for some $q \in \bbZ_+$ and $-q + 1 \le n \le q$, $z \in \bbC$, and $u$ solves \eqref{eq:diffEq}.
Then, we have
\begin{equation}\label{f.gordonest}
\max \big\{ \left\| \mathbf{U}(-q) \right\| ,
\left\| \mathbf{U}(q) \right\| ,
\left\| \mathbf{U}(2q) \right\| \big\} \ge \frac12 \left\| \mathbf{U}(0) \right\|.
\end{equation}
\end{lemma}

\begin{proof}
By assumption, we have
$$
\mathbf{U}(2p)
=
M_E(2q) \cdot \mathbf{U}(0)
=
M_E(q)^2 \cdot \mathbf{U}(0)
$$
and similarly
$$
\mathbf{U}(q)
=
M_E(q)^2 \cdot \mathbf{U}(-q).
$$
Moreover, the Cayley-Hamilton theorem implies
$$
M_E(q)^2 - \tr( M_E(q)) \cdot M_E(q) + I = 0.
$$
Consequently, we have
\begin{equation}\label{f.gordoncase1}
\mathbf{U}(2q)
-  \tr( M_E(p)) \cdot  \mathbf{U}(p)
+ \mathbf{U}(0)
= \begin{bmatrix} 0 \\ 0
\end{bmatrix}
\end{equation}
and
\begin{equation}\label{f.gordoncase2}
\mathbf{U}(q)
-  \tr( M_E(q)) \mathbf{U}(0)
+ \mathbf{U}(-q)
= \begin{bmatrix} 0 \\ 0
\end{bmatrix}.
\end{equation}
The assertion \eqref{f.gordonest} follows from \eqref{f.gordoncase1} when $|\tr( M_E(q))| \le 1$ and it follows from \eqref{f.gordoncase2} when $|\tr( M_E(q))| > 1$.
\end{proof}

The estimate \eqref{f.gordonest} can of course be used to exclude the existence of decaying solutions.

\begin{theorem} \label{t:threeblock:gordon}
Suppose that there exist $q_k \to \infty$ such that
\[
V(n - q_k) = V(n) = V(n + q_k)
\text{ for all } 1 \le n \le q_k.
\]
Then $H$ has purely continuous spectrum.
\end{theorem}

\begin{proof}
Let $E$ be given, and let $u$ denote a nontrivial solution to \eqref{eq:diffEq}, then Lemma~\ref{l.gordon} implies that
\[
\max( \|\mathbf U(-q_k)\|, \|\mathbf U(q_k)\|, \|\mathbf U(2q_k)\|)
\ge
\frac{1}{2} \|\mathbf U(0)\|.
\]
Thus, $u$ cannot go to zero at $\pm \infty$, so $u \notin \ell^2(\bbZ)$. Consequently, $z$ is not an eigenvalue of $H$.
\end{proof}

Notice that the energy $E$ plays no role in the argument. As soon as the potential $V$ has the required local periodicity for infinitely many values of $q$, we have the estimate \eqref{f.gordonest} for infinitely many values of $q$, which in turn shows that no $E$ can be an eigenvalue.

It is clear that one can perturb about this situation a little bit and still deduce useful estimates. In light of this, the following definition is natural.

\begin{definition}
A bounded potential $V : \bbZ \to \bbR$ is called a \emph{Gordon potential} if there are positive integers $q_k \to \infty$ such that
\begin{equation}\label{f.gordondef1}
\max_{1 \le n \le q_k} |V(n) - V(n \pm q_k)| \le k^{-q_k}
\end{equation}
for every $k \ge 1$. Equivalently, there are positive integers $q_k \to \infty$ such that
\begin{equation}\label{f.gordondef2}
\forall \, C > 0 : \lim_{k \to \infty} C^{q_k} \max_{1 \le n \le q_k}
|V(n) - V(n \pm q_k)|
=
0.
\end{equation}
\end{definition}

\begin{theorem} \label{t.gordon.contspec}
If $V$ is a Gordon potential, then $H$ has purely continuous spectrum. More precisely, for every $E \in \bbC$ and every solution $u$ of \eqref{eq:diffEq1}, we have
\begin{equation}\label{f.gordonpotsolest}
\limsup_{|n| \to \infty} \left\| \mathbf{U}(n) \right\| \ge \frac12 \left\| \mathbf{U}(0) \right\|,
\end{equation}
where $\mathbf{U}$ is defined by \eqref{d:2by2:u2}.
\end{theorem}

\begin{proof}
By assumption, there is a sequence $q_k \to \infty$ such that \eqref{f.gordondef2} holds. Given $E \in \bbC$, we consider a solution $u$ of \eqref{eq:diffEq} and, for every $k$, a solution $u_k$ of
$$
u_k(n+1) + u_k(n-1) + V_k(n) u_k(n) = E u_k(n)
$$
with $u_k(1) = u(1)$ and $u_k(0) = u(0)$, where $V_k$ is the $q_k$-periodic potential that coincides with $V$ on the interval $1 \le n \le q_k$. It follows from Lemma~\ref{l.gordon} that $u_k$ satisfies the estimate
$$
\max \big\{
\left\|\mathbf{U}_k(-q_k) \right\| ,
\left\|\mathbf{U}_k(q_k)  \right\| ,
\left\|\mathbf{U}_k(2q_k) \right\|
\big\} \ge \frac12 \left\| \mathbf{U}(0) \right\|,
$$
where $\mathbf U_k(n) = (u_k(n+1), u_k(n))^\top$, as usual. Since $V$ is very close to $V_k$ on the interval $[-q_k+1,2q_k]$ and $u$ and $u_k$ have the same initial conditions, we expect that they are close throughout this interval and hence $u$ obeys a similar estimate.

Let us make this observation explicit. Denote the transfer matrices associated with $V_k$ by $M_{k,E}(n)$. We have
\begin{align*}
\max_{-q_k \le n \le 2q_k}
\left\| \mathbf{U}(n) - \mathbf{U}_k(n) \right\|
& \le
\max_{-q_k \le n \le 2q_k}
\left\| M_E(n) - M_{k,E}(n) \right\| \left\| \mathbf{U}(0) \right\| \\ & \le C^{q_k} \max_{-q_k + 1 \le n \le 2 q_k} |V(n) - V_k(n)| \left\|
\mathbf{U}(0) \right\|,
\end{align*}
which goes to zero by \eqref{f.gordondef2}.
\end{proof}

We are now ready to apply the Gordon criterion in the context of limit-periodic potentials; compare the paper \cite{DG11a} by Damanik and Gan.

\begin{prop} \label{p.lp.gordon.generic}
There exists a dense $G_{\delta}$ set $\mathcal{G} \subseteq C(\Omega, \bbR)$ such that the potential $V_{\omega,f}$ defined by \eqref{e.lpdynomega} is a Gordon potential for all $f \in \mathcal{G}$ and all $\omega \in \Omega$. That is, $V_{\omega,f}$ satisfies \eqref{f.gordondef1} for suitable positive integers $q_k \to \infty$.
\end{prop}

\begin{proof}
Define $\varepsilon_n = n^{-n}$ and
$$
\mathcal{O}_k
=
\bigcup_{q = 1}^\infty \bigcup_{f \in P_q} B\left(f, \frac{\varepsilon_{kq}}{2} \right), \quad k \in \bbZ_+,
$$
where $P_q \subset C(\Omega,\bbR)$ denotes the set of $q$-periodic sampling functions defined over $\Omega$. Since $\mathcal O_k$ is clearly open, it follows that
$$
\mathcal{G} := \bigcap_{k=1}^{\infty} \mathcal{O}_k
$$
is a $G_\delta$ in $C(\Omega,\bbR)$. One can easily verify that $\mathcal{G}$ contains all periodic sampling functions, so, since $\mathcal{P}$ is dense by Theorem~\ref{t:pdense}, $\mathcal{G}$ is dense as well.

We next show that all elements of $\mathcal{G}$ produce Gordon potentials. To that end, let $f \in \mathcal{G}$ and $\omega \in \Omega$ be given. By definition, $f \in \mathcal{O}_k$ for each $k$, so we may produce a sequence $f_k \in \mathcal{P}$ such that $f_k $ is $\ell_k$-periodic for some $\ell_k \in \bbZ_+$ and $\| f - f_k \|_{\infty} < \frac{1}{2} \varepsilon_{k \cdot \ell_k}$ for each $k$. Set $q_k = k \cdot \ell_k$ and note that $q_k \geq k$, whence $q_k \to \infty$. For all $k$, $n$ such that $1 \leq n \leq q_k$, we use the triangle inequality and periodicity to obtain
\begin{align*}
| V_{\omega}(n) - V_{\omega}(n \pm q_k) |
& =
\left| f(T^n \omega) - f(T^{n \pm q_k} \omega) \right| \\
& \leq
\left| f(T^n \omega) - f_k(T^{n} \omega) \right| + \left| f_k(T^n \omega) - f(T^{n \pm q_k} \omega) \right| \\
& =
\left| f(T^n \omega) - f_k(T^{n} \omega) \right| + \left| f_k(T^{n \pm k \ell_k} \omega) - f(T^{n \pm q_k} \omega) \right| \\
& <
\varepsilon_{k \cdot \ell_k} \\
& \le
k^{-q_k}.
\end{align*}
Taking the max over $n$ with $1 \leq n \leq q_k$ yields
$$
\max_{1 \leq n \le q_k} |V(n) - V(n \pm q_k)|
\le
k^{-q_k},
$$
and hence $V_{\omega} = V_{\omega,f}$ is a Gordon potential for all $\omega \in \Omega$ and all $f \in \mathcal{G}
$.
\end{proof}

\begin{proof}[Proof of Theorem \ref{t.lp.scspec.generic}]
Let $\mathcal{C}, \, \mathcal{G} \subseteq C(\Omega,\bbR)$ be the sets constructed in Theorem~\ref{t.lp.cantorspec.generic} and Proposition~\ref{p.lp.gordon.generic} respectively, and put $\mathcal{S} = \mathcal{C} \cap \mathcal{G}$. By the Baire Category Theorem, $\mathcal{S}$ is a dense $G_{\delta}$.  Since $\mathcal{S} \subseteq \mathcal{C}$, the zero-measure Cantor set statement follows.

For all $f \in \mathcal{S}$ and $\omega \in \Omega$, the Lebesgue measure of $\sigma(H_{\omega})$ is zero, which precludes the presence of absolutely continuous spectrum.  Similarly, for every $f \in \mathcal{S}$ and $\omega \in \Omega$, $V_{\omega}$ is a Gordon potential, so by Theorem~\ref{t.gordon.contspec}, it follows that $H_{\omega}$ has no eigenvalues. Thus, $H_\omega$ has purely singular continuous spectrum, as claimed.
\end{proof}

\section{Pure Point Spectrum}\label{s.pp}

We have seen that purely absolutely continuous spectrum and purely singular continuous spectrum are each dense phenomena in the space of limit-periodic Schr\"odinger operators, with the latter being even generic. This raises the natural question of what can be said about cases with pure point spectrum. It turns out that even this spectral type is a dense phenomenon. Clearly this is surprising as in these cases the spectral type of the limit object is as different from the spectral type of the approximants as it can be.

\begin{theorem}\label{t.lppp1}
For every limit-periodic $V$ and every $\varepsilon > 0$, there is a limit-periodic $\widetilde V$ with $\|V - \widetilde V\|_\infty < \varepsilon$ such that the Schr\"odinger operator with potential $\widetilde V$ has pure point spectrum.
\end{theorem}

Note that we do not fix the base dynamics and vary the sampling function only. Rather we perturb in the space of all limit-periodic potentials. It would be interesting to prove a variant of Theorem~\ref{t.lppp1} with fixed base dynamics.

The proof of Theorem~\ref{t.lppp1} actually does not really care about the fact that the potential $V$ we are perturbing is limit-periodic. Indeed, the proof directly yields the following more general result, which is due to Damanik and Gorodetski \cite{DG16}:

\begin{theorem}\label{t.lppp2}
For every bounded $V$ and every $\varepsilon > 0$, there is a limit-periodic potential $V_\mathrm{lp}$ with $\|V_\mathrm{lp}\|_\infty < \varepsilon$ such that the Schr\"odinger operator with potential $V + V_\mathrm{lp}$ has pure point spectrum.
\end{theorem}

In other words, a suitable arbitrarily small limit-periodic perturbation can turn any given spectral type into one that is pure point. Such a strong dominance property was previously only known for small random perturbations. Not coincidentally, the proof of Theorem~\ref{t.lppp2} is based on a generalization of the argument that yields the latter statement. Thus, let us first explain how pure point spectrum can be proved for random potentials, then for perturbations of fixed background potentials by arbitrarily small random potentials, and finally for the setting in question -- perturbations of fixed background potentials by arbitrarily small limit-periodic potentials.

The way we want to explain why random potentials lead to pure point spectrum is based on the Kunz-Souillard method \cite{KS80}. Suppose $V_\omega$ is a bounded random potential and $H_\omega$ is the associated Schr\"odinger operator. Integration with respect to the underlying probability measure $\mu$ will be denoted by $\mathbb{E}(\cdot)$. A key quantity in this approach is
$$
a(n,m) = \mathbb{E} \left( \sup_{t\in \bbR} \left| \left\langle \delta_n , e^{-itH_{\omega}} \delta_m \right\rangle \right| \right), \quad n,m \in \bbZ.
$$
The following sufficient criterion for (almost sure) pure point spectrum follows from the RAGE theorem:

\begin{prop}\label{p.asppcriterion}
Suppose that
$$
\sum_{n \in \bbZ} |a(n,m)| < \infty
$$
for $m = 0,1$. Then, for $\mu$-almost every $\omega$, the operator $H_\omega$ has pure point spectrum.
\end{prop}

For $L \in \bbZ_+$, we denote by $H_{\omega}^{(L)}$ the restriction of $H_{\omega}$ to $\ell^2(\{-L,\ldots, L\})$ with Dirichlet boundary conditions. For $|n|, |m| \le L$, we define $a_L(n,m)$ to be $a(n,m)$ with $H_{\omega}$ replaced by $H_{\omega}^{(L)}$, that is,
$$
a_L(n,m) = \mathbb{E} \left( \sup_{t\in \bbR} \left| \left\langle \delta_n , e^{-itH_{\omega}^{(L)}} \delta_m \right\rangle \right| \right).
$$
It is easy to see that $H_{\omega}^{(L)}$ has $2L+1$ real simple eigenvalues
$$
E_{\omega}^{L,1} < E_{\omega}^{L,2} < \cdots < E_{\omega}^{L,2L+1}.
$$
By our boundedness assumption there exists a compact interval $\Sigma_0$ that contains all eigenvalues $E_{\omega}^{L,k}$. For each $k$, let $\varphi_{\omega}^{L,k}$ denote a normalized eigenvector corresponding to $E_{\omega}^{L,k}$. We now define
$$
\rho_L(n,m) = \mathbb{E} \left( \sum_{k=1}^{2L+1} \left| \left\langle \delta_n, \varphi_{\omega}^{L,k} \right\rangle \right| \left| \left\langle \delta_m , \varphi_{\omega}^{L,k} \right\rangle \right| \right).
$$

\begin{lemma} \label{l.ks.est1}
{\rm (a)} For $n,m \in \bbZ$, we have
\begin{equation}\label{e.ksest1}
a(n,m) \leq \liminf_{L \to \infty} a_L(n,m).
\end{equation}
{\rm (b)} If $|n|,|m| \leq L$, then
\begin{equation}\label{e.ksest2}
a_L(n,m) \leq \rho_L(n,m).
\end{equation}
\end{lemma}

Part (a) follows from strong approximation and Fatou's lemma and part (b) follows by simply expanding $\delta_m$ in the basis of eigenvectors and using that the eigenvalues are real. It follows that we can estimate $a(n,m)$ from above by proving upper bounds for $\rho_L(n,m)$ that are uniform in $L$.

Let us now assume, for the time being, that the potential of the random Schr\"odinger operator $H_\omega$ is generated by independent identically distributed random variables, each of which has a bounded compactly supported density $r$. Thus, we can write
\begin{equation}\label{e.rhoLnmexp}
\rho_L(n,m)
=
\int \cdots \int \left( \sum_{k=1}^{2L+1} \left| \left\langle \delta_n, \varphi_{v_{-L},\ldots,v_L}^{L,k} \right\rangle \right| \left| \left\langle \delta_m , \varphi_{v_{-L},\ldots,v_L}^{L,k} \right\rangle \right| \right) \left( \prod_{j = -L}^L r(v_j) \right) \, dv_{-L} \ldots dv_L,
\end{equation}
since $H_{\omega}^{(L)}$ only depends on the random variables corresponding to $|j| \le L$. Moreover, it also follows that $\rho_L(n,m) = \rho_L(m,n)$ and $\rho_L(n,m) = \rho_L(n-m,0)$, so that it suffices to estimate $\rho_L(n,m)$ for $n \ge 0$ and $m = 0$.

Let us apply a suitable change of variables to the iterated integral expressing $\rho_L(n,0)$, $n \ge 0$. If $E$ is $E_{v_{-L},\ldots,v_L}^{L,k}$ and $u$ is $\varphi_{v_{-L},\ldots,v_L}^{L,k}$, then we have
\begin{equation}\label{e.kseve1}
u(\ell+1) + u(\ell-1) + v_\ell \, u(\ell) = E u(\ell)
\end{equation}
for $-L \le \ell \le L$, where $u(-L-1) = u(L+1) = 0$. We rewrite this identity as
\begin{equation}\label{e.kseve2}
v_\ell = E - \frac{u(\ell+1)}{u(\ell)} - \frac{u(\ell-1)}{u(\ell)},
\end{equation}
and this in turn motivates the change of variables
$$
\{ v_\ell \}_{\ell = -L}^L \quad \longleftrightarrow \quad \{ x_{-L} , \ldots , x_{-1} , E , x_1 , \ldots , x_L \},
$$
where
$$
E = E_{v_{-L},\ldots,v_L}^{L,k}
$$
and
\begin{equation}\label{e.kseve3}
x_\ell
=
\begin{cases}
\frac{\varphi_{v_{-L},\ldots,v_L}^{L,k} (\ell+1)}{\varphi_{v_{-L},\ldots,v_L}^{L,k} (\ell)} & \ell < 0 \\
\frac{\varphi_{v_{-L},\ldots,v_L}^{L,k} (\ell-1)}{\varphi_{v_{-L},\ldots,v_L}^{L,k} (\ell)} & \ell > 0,
\end{cases}
\end{equation}
so that
\begin{equation}\label{e.kseve4}
v_\ell = \begin{cases} E - x_{\ell-1}^{-1} - x_\ell & \ell < 0 \\ E - x_{-1}^{-1} - x_1^{-1} & \ell = 0 \\ E - x_{\ell+1}^{-1} - x_\ell & \ell > 0, \end{cases}
\end{equation}
with the conventions $x_{-L-1}^{-1} = x_{L+1}^{-1} = 0$ (which are natural in view of the definitions above).

Implementing this change of variables, the $2L+1$-fold iterated integral expressing $\rho_L(n,0)$ takes the following form:

\begin{lemma}\label{l.ks1}
Fix $L \in \bbZ_+$ and $n$ with $0< n \leq L$. Set
\begin{align*}
\phi_E(x) & = r(E - x), \\
\left( U_0 f \right)(x) & = |x|^{-1} f \left( |x|^{-1} \right), \\
\left( S_E  f \right)(x) & =  \int_{\bbR} \! r  \left( E - x - y^{-1} \right) f(y) \, dy, \\
\left( T_E  f \right)(x) & =  \int_{\bbR} \! r  \left( E - x - y^{-1} \right) |y|^{-1} f(y) \, dy.
\end{align*}

Then, we have
\begin{equation}\label{e.ksest3}
\rho_L(n,0)
=
\int_{\Sigma_0} \left\langle T_{E}^{n-1} S_{E}^{L-n} \phi_E , U S_{E}^{L} \phi_E \right \rangle_{L^2(\bbR)} \, dE.
\end{equation}
\end{lemma}

In this representation one then proceeds by applying the Cauchy-Schwarz inequality inside the integral and applying suitable norm estimates for the operators involved that are uniform in $E \in \Sigma_0$. Specifically, the following estimates, where we denote the norm of an operator $T : L^p(\bbR) \to L^q(\bbR)$ by $\|T\|_{p,q}$, can be proved and then used in this way:

\begin{lemma} \label{l:KS:11norm}
{\rm (a)} For every $E \in \bbR$, we have $\| S_E \|_{1,1} \leq 1$. \\[1mm]
{\rm (b)} For every $E \in \bbR$, we have $\| S_E \|_{1,2} \leq \|r\|_\infty^{1/2}$. \\[1mm]
{\rm (c)} For every $E \in \bbR$, we have $\| T_E \|_{2,2} \leq 1$. \\[1mm]
{\rm (d)} There exists $q < 1$ such that for every $E \in \Sigma_0$, we have $\| T_E^2 \|_{2,2} \leq q$.
\end{lemma}

With these estimates one can show exponential decay of $\rho_L(n,0)$ and hence Proposition~\ref{p.asppcriterion} is clearly applicable.

Before turning our attention to limit-periodic perturbations and the idea underlying the proof of Theorem~\ref{t.lppp2}, we first point out that the argument sketched above can be modified, and in fact generalized, in a number of ways.

First of all, while independence is crucially used, it is not necessary to consider identically distributed random variables. Indeed, if the density $r$ changes from site to site, one simply replaces $r$ by an $\ell$-dependent density $r_\ell$ as one processes the random variable at site $\ell$. For simplicity we will consider rescalings of a fixed density: $r_\ell(x) = d_\ell^{-1} r(d^{-1}_\ell x)$ for suitable $d_\ell > 0$. This results in obvious changes to \eqref{e.rhoLnmexp} and Lemma~\ref{l.ks1}. The necessary change to Lemma~\ref{l:KS:11norm} is trickier. Namely, part (d) has no obvious replacement and needs to be modified via a quantitative estimate that is based on the following key lemma \cite{Simon1982CMP}:

\begin{lemma}\label{l.ks4}
For $\lambda \ge 0$ sufficiently small, $\widehat r$, the Fourier transform of $r$, obeys
\begin{equation} \label{eq:KS:FourierdecayAwayFrom0}
\sup_{|\eta| \ge \lambda} |\widehat r(\eta)|^2 \le e^{- c |\lambda|^2}
\end{equation}
for a suitable constant $c = c(r) > 0$. Furthermore, there exist constants $K_0 = K_0(r)$ and $\lambda = \lambda(r)$ such that for every $\ell$ and every $E \in \Sigma_0$, we have
\begin{equation} \label{eq:KS4:bdAwayFr1}
\begin{split}
\|T_E^{(\ell)} T_E^{(\ell-1)}\|_{2,2}
& \le
\frac{1}{4}
\left( 15
+ \sup_{|\eta| \ge K_0 \min\{d_{\ell}, d_{\ell-1}\}} \left|\widehat{r}(\eta)\right|^2 \right)^{1/2} \\
& \le
e^{- c K_0^2 \min\{d_{\ell}^2, d_{\ell-1}^2, \lambda \}}.
\end{split}
\end{equation}
\end{lemma}

With this replacement of Lemma~\ref{l:KS:11norm}.(d) one can still prove that Proposition~\ref{p.asppcriterion} is applicable under suitable assumptions on the sequence $\{ d_\ell \}$ (which obviously cannot decay too fast since the modification of Lemma~\ref{l:KS:11norm}.(b) results in a factor $d_\ell^{-1/2}$ and the influence of the exponential in \eqref{eq:KS4:bdAwayFr1} becomes weaker with smaller $d_{\ell}, d_{\ell-1}$).

The second generalization of the basic argument outlined above is the addition of a fixed bounded background potential. Indeed, this simply results in obvious changes to \eqref{e.kseve1}--\eqref{e.kseve4} and then to Lemma~\ref{l.ks1}. Note that this generalization produces the precursor to Theorem~\ref{t.lppp2}, which says that any fixed bounded potential can be turned into one whose associated Schr\"odinger operator has pure point spectrum by an arbitrarily small \emph{random} perturbation.

The third generalization is the most severe. It is here that the independence of the values of the perturbation potential is given up, which is clearly necessary if we are to produce limit-periodic perturbations.

To explain what kind of correlations we wish to introduce, let us discuss the general setup of the proof of Theorem~\ref{t.lppp2}. The limit-periodic perturbation $V_\mathrm{lp}$ will be obtained as a convergent series
$$
V_\mathrm{lp} = \sum_{k=1}^\infty V_\mathrm{per}^{(k)}
$$
where $V_\mathrm{per}^{(k)}$ is $(2q_k + 1)$-periodic and we have
$$
\sum_{k=1}^\infty \|V_\mathrm{per}^{(k)}\|_\infty < \varepsilon
$$
in order to satisfy $\|V_\mathrm{lp}\|_\infty < \varepsilon$. The values of $V_\mathrm{per}^{(k)}$ on $\{-q_k, -q_k+1, \ldots, q_k - 1, q_k\}$ will be generated by independent identically distributed random variables with density $r_k(x) = d_k^{-1} r(d^{-1}_k x)$, where $r$ is a fixed density and $d_k > 0$ is chosen suitably. Note that for each $\ell$, there will be infinitely many random variables (with scaling factors $d_{k}$, $k \ge k_0(\ell)$) participating in determining the value of $V_\mathrm{lp}$ at site $\ell$. The key idea will be to choose a decaying sequence $\{d_k\}$ and to consider the lowest-level participating random variable with scaling factor $d_{k_0(\ell)}$ as the \emph{essential} random variable at site $\ell$, and all others as \emph{inessential} random variables that are frozen and put in the background potential for the purpose of estimating the relevant quantities as above \emph{uniformly} in the background potential. One can then average the resulting estimate over the inessential random variables to obtain the desired estimate of the overall expectation defining $\rho_L(n,0)$. Of course the periodicity of $V_\mathrm{per}^{(k)}$ results in another change to the equations \eqref{e.kseve1}--\eqref{e.kseve4} and then to Lemma~\ref{l.ks1}.

This strategy may be implemented in a way that allows one to choose the sequences $\{q_k\}$ (determining the periods of the $V_\mathrm{per}^{(k)}$'s) and $\{d_k\}$ (determining the sizes of the $V_\mathrm{per}^{(k)}$'s) so that Proposition~\ref{p.asppcriterion} may be applied, showing that almost all of these limit-periodic perturbations will produce operators with pure point spectrum. Choosing one of them then establishes Theorem~\ref{t.lppp2}. We refer the reader to \cite{DG16} for further details. Note that in this line of reasoning one only gets the pure point property for the $V_\mathrm{lp}$ as constructed (chosen from the full measure set), but not for the elements of its hull. As a consequence, as pointed out above, the formulation of the result is different from some of the earlier results about the spectral type and does not apply uniformly across the hull of the limit-periodic sequence in question.

\bigskip

It is however possible to prove pure point spectrum uniformly across the hull in certain scenarios (which are not dense). This follows from a combination of works of P\"oschel \cite{P83} and Damanik-Gan \cite{DG11}.

First, we discuss P\"oschel's work contained in \cite{P83}. Given a function $\Omega:[0,\infty) \to [0,\infty)$, define the quantities $\Phi_\Omega(t)$, $\kappa$, and $\Psi_\Omega(t)$ for $t>0$ by
\begin{align*}
\Phi_\Omega(t)
& :=
t^{-4} \sup\{\Omega(r) \, e^{-t r} : r \geq 0\} \\
\kappa_t
& :=
\set{\{t_j\}_{j = 0}^\infty : t \geq t_0 \geq t_1 \geq \cdots  \text{ and } \sum_{j = 0}^\infty t_j \leq t} \\
\Psi_\Omega(t)
& :=
\inf_{(t_j) \in \kappa_t} \prod_{j = 0}^\infty \Phi(t_j)2^{-j-1}
\end{align*}
for $t > 0$. We call $\Omega$ an \emph{approximation function} if $\Phi_\Omega(t)$ and $\Psi_\Omega(t)$ are finite for every $t > 0$. For example, for any $\alpha \geq 0$, $\Omega(r) = r^\alpha$ defines an approximation function.

Now, suppose that $\mathcal{M}$ is a Banach subalgebra of $\ell^\infty(\bbZ)$ with respect to the operations of pointwise addition and pointwise multiplication; in particular, we assume that $\mathbbm{1}$ is contained in $\mathcal{M}$ (where $\mathbbm{1}(n) \equiv 1$). We say that $\lambda : \bbZ \to \bbR$ is a \emph{distal sequence for} $\mathcal{M}$ if, for each $k \in \bbZ\setminus\{0\}$, one has $\big( \lambda - S^k\lambda \big)^{-1} \in \mathcal{M}$, and one has the bound
\[
\left\| \big(\lambda - S^k \lambda\big)^{-1} \right\|_\infty
\leq
\Omega(|k|),
\quad
\text{for all } k \neq 0,
\]
where $\Omega$ is an approximation function. Here, the inverse refers to the multiplicative inverse in the Banach algebra.

P\"oschel proved the following theorem in \cite{P83}:

\begin{theorem}\label{t.poeschel}
If $\lambda$ is a distal sequence for the Banach algebra $\mathcal{M}$, then there exists $\varepsilon_0 > 0$ such that the following holds true. For any $0< \varepsilon \leq \varepsilon_0$, there is a sequence $V$ so that $\lambda - V \in \mathcal{M}$ and the Schr\"odinger operator $H_{\varepsilon^{-1}V} = \Delta + \varepsilon^{-1} V$ is spectrally localized with eigenvalues $\{ \varepsilon^{-1} \lambda_j : j \in \bbZ\}$. Moreover, if $\psi_k$ is the normalized eigenvector corresponding to the eigenvalue $\lambda_k$, then there are constants $c>0$ and $d>1$ such that
\[
|\psi_k(n)|^2
\leq
c d^{-|k-n|}
\]
for all $k$ and $n$.
\end{theorem}

He also supplied some examples. Here is one that is particularly interesting because it shows not only that the theorem above can be applied and yields uniformly localized limit-periodic Schr\"odinger operators, but also that the spectrum in this case has no gaps (which is surprising in view of our earlier results showing that Cantor spectra are typical for limit-periodic Schr\"odinger operators):

\medskip

\noindent \textbf{P\"oschel's Example: A Limit-Periodic Distal Sequence.} Let $\mathcal{D}_n$ denote the set of sequences in $\ell^\infty(\bbZ)$ having period $2^n$, and $\mathcal{D} = \bigcup_n \mathcal{D}_n$; the space
\[
\mathcal{L}
=
\overline{\mathcal{D}}
\]
is a Banach algebra and a subspace of the space of all limit-periodic sequences. Let us describe how to construct a distal sequence for $\mathcal{L}$. For $j \in \bbZ_+$, define the set $B_j$ by
\[
B_j
:=
\begin{cases}
\displaystyle\bigcup_{N \in \bbZ} [N \cdot 2^j, N\cdot 2^j + 2^{j-1}), & \text{ if } j \text{ is even};\\[3.5mm]
\displaystyle\bigcup_{N \in \bbZ} [N \cdot 2^j+2^{j-1}, (N+1)\cdot 2^j), &  \text{ if } j \text{ is odd}.
\end{cases}
\]
For example, $B_1$ is precisely the set of odd integers. Denoting the indicator function of $B_j$ by $b_j = \chi_{B_j}$, define $\lambda$ by
\[
\lambda_n
=
\sum_{j=1}^\infty b_j(n) 2^{-j}.
\]
It is immediate from the definition that $\lambda \in \mathcal{L}$; moreover, the inequality
\[
\left\| \big(\lambda - S^k\lambda \big)^{-1} \right\|
\leq
16|k|
\]
for $k \neq 0$ means that $\lambda$ is distal for $\mathcal{L}$. It is not too difficult to prove the following statement:

\begin{lemma} \label{l:dyadicLanding}
For any $m \in \bbZ_+$ and any integer $0 \leq j < 2^m$, there is an integer $\ell = \ell(j,m)$ so that
\[
\lambda_k \in I_{m,j} := \left[\frac{j}{2^m}, \frac{j+1}{2^m} \right)
\iff
k \in \ell + 2^m \bbZ.
\]
\end{lemma}

We see that in this particular example, the Schr\"odinger operators produced by Theorem~\ref{t.poeschel} have spectrum $[0,\varepsilon^{-1}]$.

Damanik and Gan showed in \cite{DG11} that P\"oschel's results extend to the hull, that is, whenever Theorem~\ref{t.poeschel} can be applied to produce limit-periodic potentials for which the associated Schr\"odinger operator is uniformly localized, then the same statement is true, with the same constants, for all elements of the hull of the potential in question; see also \cite{H16} for a generalization of this statement.

\section{The Density of States}\label{s.ids}

Recall that we defined the density of states (DOS) measure $dk$ to be the weak limit of $dk_N$, where
\[
\int g \, dk_N
=
\frac{1}{N} \tr(P_N \, g(H) \, P_N^*),
\]
for Borel sets $B$, and $P_N$ denotes projection onto coordinates $\{0,\ldots,N-1\}$ (provided that the limit exists). In the event that $V$ is limit-periodic, we saw that the DOS exists by unique ergodicity. The accumulation function of the DOS is called the \emph{integrated density of states} (IDS) and is denoted by
\[
k(E)
=
\int \chi_{(-\infty,E]} \, dk.
\]

The results of the present section are concerned with the regularity of $k$ as a function of $E$. In full generality, $dk$ is continuous (Theorem~\ref{t:DelSou}). In light of this, it is natural to ask whether one has a quantitative modulus of continuity, for example $\alpha$-H\"older continuity for some $\alpha > 0$. In full generality, this is too ambitious, but one can wring just a bit more continuity out of the Thouless formula, as the following theorem of Craig and Simon \cite{CS83CMP} illustrates:

\begin{theorem}
For any almost-periodic potential, the integrated density of states is log-H\"older continuous. That is, there is a constant $C>0$ with the property that
\[
|k(E_1) - k(E_2)|
\leq
C (\log|E_1 - E_2|^{-1})^{-1}
\]
for all $E_1,E_2 \in \bbR$ with $|E_1 - E_2| \leq 1/2$.
\end{theorem}

\begin{proof}
Without loss of generality, assume $E_1 < E_2 \leq E_1 + \frac{1}{2}$. Then, since all the transfer matrices have determinant one, we have $L(E_1) \geq 0$, which leads to
\begin{align*}
0 & \le L(E_1) \\
& =
\int \log | E - E_1 | \, dk(E) \\
& =
\int_{(E_1,E_2)} \log | E - E_1 | \, dk(E) + \int_{\bbR\setminus (E_1,E_2)} \log | E - E_1 | \, dk(E)
\end{align*}
Rearranging, we get
$$
- \int_{(E_1,E_2)} \log | E - E_1 | \, dk(E) \le \int_{\bbR\setminus (E_1,E_2)} \log | E - E_1 | \, dk(E).
$$
Bounding the integrands of each side, we obtain
$$
- \log | E_2 - E_1 | \int_{(E_1,E_2)}  \, dk(E)
\le
\log ( | E_1 | + \|f\|_\infty + 2 ) \int \, dk(E),
$$
since $\supp(dk) \subseteq [-2 - \|f\|_\infty, 2 + \|f\|_\infty]$. Thus, with $C= \log(|E_1| + \|f\|_\infty + 2)$, we have
\[
|k(E_2) - k(E_1)|
=
\int_{(E_1,E_2)} dk(E)
\leq
C \left[\log|E_1 - E_2|^{-1}\right]^{-1}.
\]
\end{proof}

Within the class of limit-periodic potentials, this result is optimal; in particular, the following result of Kr\"uger and Gan shows that there is a dense set of limit-periodic potentials whose potentials are not $h$-H\"older continuous for any function $h$ that goes to zero faster than $[\log(1/\delta)]^{-1}$ \cite{KruGan}.

\begin{theorem}
Let $\Omega$ denote a Cantor group with a minimal translation $T$. There is a dense set $\mathcal{I} \subseteq C(\Omega,\bbR)$ with the property that the density of states of $V(n) = f(T^n 0 )$ is no better than log-H\"older continuous. That is, if $k$ denotes the IDS of $V$, given any increasing function $h:\bbR_+ \to \bbR_+$ with
\[
\lim_{\delta \downarrow 0} h(\delta) \log(1/\delta)
=
0,
\]
one has
\[
\limsup_{E \to E_0} \frac{|k(E) - k(E_0)|}{h(|E - E_0|)}
=
\infty
\]
for at least one $E_0$. In fact, one can guarantee that this holds for all $E_0$ in the spectrum.
\end{theorem}

The main idea is the following: since $\CP$  is dense, start with some $f_0 \in \CP$ and let $\varepsilon>0$. Then, inductively choose $f_1,\ldots \in \CP$ using Lemma~\ref{l:lp:smallspec} so that $\|f_j - f_{j-1}\|_\infty < \varepsilon \cdot 2^{-j}$ and so that $\Leb(\Sigma(f_j))$ is exponentially small, that is
\[
\Leb(\Sigma(f_j))
\lesssim
e^{-c_j q_j}
\]
where $q_j$ is the period of $f_j$. Then, by completeness, $f_\infty = \lim f_j$ exists. The density of states corresponding to the potential function $f_j$ gives weight $1/q_j$ to an interval having length no greater than $\exp(-c_j q_j)$. By carefully tuning the rate at which $f_j \to f_\infty$, one can push a statement like this through to the DOS of $f_\infty$.
\bigskip

On the other hand, for a dense set, namely, for $f \in \mathcal{P}$, one knows that the IDS is 1/2-H\"older continuous (and no better); compare Theorem~\ref{t.periodic.ids}. It is interesting to ask whether one can do better than 1/2. In the localization regime of P\"oschel, the IDS can be shown to be 1-H\"older continuous, which was proved by Damanik and Fillman \cite{DF2018}.

\begin{theorem}
There exist limit-periodic $V$ whose associated IDS is Lipschitz-continuous.
\end{theorem}

\section{Other Families of Operators}\label{s.other}

Let us explore some models that are closely connected with, but distinct from the discrete 1D Schr\"odinger operators that we have considered thus far. For each of these families of models, many of the results from previous sections have analogs. To avoid being tedious, we will be somewhat selective with our presentation in this section, focusing on results whose analogs involve interesting challenges.

\subsection{Jacobi Matrices}

A \emph{Jacobi matrix} is an operator on $\ell^2(\bbZ)$ of the form
\[
[\sJ u]_n
=
a_{n-1} u_{n-1} + b_n u_n + a_n u_{n+1},
\quad
n \in \bbZ, \; u \in \ell^2(\bbZ),
\]
where $a_n, b_n \in \bbR$ with
\[
\sup_n |b_n| < \infty,
\quad
\sup_n a_n < \infty,
\quad
\inf_n a_n > 0.
\]
Thus, discrete Schr\"odinger operators are obtained from Jacobi matrices as a special case (in which $a_n \equiv 1$). Most of the results discussed in the foregoing section have Jacobi analogs.

A periodic Jacobi matrix is one for which the parameters $a$ and $b$ satisfy $a_{n+q} = a_n$ and $b_{n+q} = b_n$ for all $n$. Thus, a limit-periodic Jacobi matrix $\sJ$ is one for which there are periodic Jacobi matrices $\{ \sJ_j \}_{j=1}^\infty$ such that
\[
\lim_{j\to\infty} \|\sJ - \sJ_j\|
=
0,
\]
where $\| \cdot \|$ denotes the operator norm.

Jacobi matrices are not only considered simply for the sake of a more general setting (even though it is the natural generalization that still keeps many of the essential features intact; notably self-adjointness and a $2 \times 2$ transfer matrix formalism) but also because they provide the natural setting for the study of inverse spectral problems, where it is not at all clear that a solution to the given problem at hand can be found in the class of discrete Schr\"odinger operators, but where it can be shown to be solvable within the class of Jacobi matrices.

The most natural instance of this principle is the association of a spectral measure to an operator, either a discrete Schr\"odinger operators or a Jacobi matrix, on the discrete half-line (i.e., an operator acting in $\ell^2(\bbZ_+)$) and the corresponding spectral measure (associated with the cyclic vector $\delta_1$). This correspondence, which can be established via the spectral theorem when passing from operator to measure and via orthogonal polynomials or a continued fraction expansion of the Borel transform of the measure when passing from measure to operator, sets up a natural bijection between bounded Jacobi matrices and compactly supported probability measures on the real line, but it continues to be an open problem to characterize explicitly those measures that correspond to discrete Schr\"odinger operators.

More closely related to the topic of this survey, however, is the way in which Jacobi matrices arise as solutions of certain renormalization equations, which can sometimes be shown to give rise to limit-periodic coefficients. Let us describe work in this spirit by Peherstorfer, Volberg and Yuditskii \cite{PVY06}.

Suppose $T$ is an expanding polynomial and denote its real Julia set by $\mathrm{Julia}(T)$ and its degree by $d$. Recall that $\mathrm{Julia}(T)$ is the compact set of real numbers that do not go to infinity under forward iterations of $T$. Under the normalization
$$
T^{-1} : [-1,1] \to [-1,1]; \pm 1 \in \mathrm{Julia}(T)
$$
such a polynomial is well-defined by the position of its critical values
$$
\mathrm{CV}(T) = \{ t_i = T(c_i) : T'(c_i) = 0, \; c_i > c_j \text{ for } i > j \}.
$$
$T$ is \emph{expanding} (or \emph{hyperbolic}) if
$$
c_i \not\in \mathrm{Julia}(T) \quad \forall i.
$$
Then we have the following pair of theorems from \cite{PVY06}:

\begin{theorem}
Let $\sJ$ be a Jacobi matrix with spectrum contained in $[-1, 1]$. With the polynomial $T$ of degree $d$ from above, consider the renormalization equation
\begin{equation}\label{e.pvyreneq}
V^* (z - \sJ)^{-1} V = (T(z) - \tilde{\sJ})^{-1} \frac{T'(z)}{d}
\end{equation}
where
$$
V \delta_k  = \delta_{dk}
$$
for each $k \in \bbZ$. It has a solution $\sJ = \sJ(\tilde{\sJ})$ with spectrum contained in $T^{-1}([-1,1])$.

Moreover, if $\min_i |t_i| \ge 10$, then
$$
\|\sJ(\tilde{\sJ}_1) - \sJ(\tilde{\sJ}_2)\| \le \kappa \| \tilde{\sJ}_1 - \tilde{\sJ}_2\|
$$
with an absolute constant $\kappa < 1$.
\end{theorem}

\begin{theorem}
Let us assume that $T$ is \emph{sufficiently hyperbolic} in the sense that
$$
\mathrm{dist}(\mathrm{CV}(T),[-1,1]) \ge 10.
$$
Then the renormalization equation \eqref{e.pvyreneq} has a unique fixed point. That is, there is a unique Jacobi matrix $\sJ$ such that
$$
V^* (z - \sJ)^{-1} V = (T(z) - \sJ)^{-1} \frac{T'(z)}{d}.
$$
Moreover, the coefficients of $\sJ$ are limit-periodic.
\end{theorem}

\begin{remark}
To be more accurate, one has to expand the notion of a Jacobi matrix slightly for the purpose of these theorems and allow the off-diagonal terms to vanish. Indeed, the central off-diagonal element $a_0$ of the fixed point $\sJ$ is zero, and hence $J$ splits into the direct sum of two half line Jacobi matrices. It turns out that spectral measure of the Jacobi matrix corresponding to the right half line in this decomposition is the balanced (equilibrium) measure on $\mathrm{Julia}(T)$. In other words, the equilibrium measure on the Julia set of a suitable hyperbolic polynomial has orthogonal polynomials with limit-periodic recursion coefficients.
\end{remark}

\subsection{Continuum Schr\"odinger Operators}

The continuum Schr\"odinger operator acts as a self-adjoint operator in $L^2(\bbR)$ via
\begin{equation} \label{eq:contopdef}
L_V y
=
-y'' + Vy,
\end{equation}
where $V:\bbR \to \bbR$ is a sufficiently nice function; for most of the present section, $V$ will be bounded and continuous. For such $V$, \eqref{eq:contopdef} defines a self-adjoint operator in $L^2(\bbR)$ in a canonical fashion. These operators enjoy a transfer matrix formalism quite similar to the one described in the discrete setting: for each $z \in \bbC$ and $x \in \bbR$, there is an $\SL(2,\bbC)$ matrix $A_z^x(V)$ such that
\[
\begin{bmatrix}
y'(x) \\ y(x)
\end{bmatrix}
=
A_z^x(V)
\begin{bmatrix}
y'(0) \\ y(0)
\end{bmatrix}
\]
whenever $y$ satisfies $L_V y = zy$. We then consider potentials that are uniform limits of continuous periodic potentials; concretely, define
\[
\mathrm{P}(\bbR)
=
\set{V \in C(\bbR) : \text{there exists  } T>0 \text{ such that } V(x+T) \equiv V(x)}
\]
and then let $\LP(\bbR)$ denote the closure of the space of periodic potentials (with respect to the topology induced by the uniform metric):
\[
\LP(\bbR)
=
\overline{\mathrm{P}(\bbR)}^{\|\cdot\|_\infty}.
\]
In particular, every $V \in \LP(\bbR)$ is continuous and uniformly almost-periodic.

 In the event that $V \in \mathrm{P}(\bbR)$, say $V(x+T) \equiv V(x)$ for some $T>0$, the formalism from Section~\ref{ss:floq} enables one to describe the spectrum of $L_V$. Concretely, the discriminant is again given by
\[
D(z)
=
\tr \, A_z^T(V),
\]
and then the spectrum of $L_V$ is again given by
\[
\sigma(L_V)
=
\set{E \in \bbR : |D(E)| \leq 2}.
\]
On the other hand, one can also consider the analog of Section~\ref{ss:bloch}; here, we consider $L(\theta) = L_{V}(\theta)$ acting via
\begin{align*}
L_{V}(\theta) y & = -y''+Vy,\\
D(L(\theta))
& =
\set{f \in H^2([0,T]) : f(T) = e^{i\theta} f(0) \text{ and } f'(T) = e^{i\theta} f'(0)}.
\end{align*}
One can then show that $L_V(\theta)$ has compact resolvent and hence a sequence of eigenvalues $\lambda_1(\theta) \leq \lambda_2(\theta) \leq \cdots$. As before, the eigenvalues of $L_V(0)$ and $L_V(\pi)$ provide the endpoints of the spectral bands. Fixing $T = \pi$ for concreteness, we can consider $V \equiv 0$ as a $\pi$-periodic potential\footnote{Recall that the subscript in $L_V$ refers to the potential, so $L_0$ refers to the free Laplacian with $V \equiv 0$.} and explicitly compute the eigenvalues of $L_0(\theta)$ for $\theta = 0,\pi$:
\[
\sigma(L_0(0))
=
\set{(2k)^2 : k =0,1,2,\ldots},
\quad
\sigma(L_0(\pi))
=
\set{(2k+1)^2 : k = 0,1,2,\ldots}.
\]
Here, all eigenvalues are doubly degenerate with the exception of $0 \in \sigma(L_0(0))$ which is simple. Thus, (counting from $k=1$ at the bottom of the spectrum) the $k$th spectral band is $[(k-1)^2,k^2]$. So, if $V$ is then a bounded $\pi$-periodic potential, the $k$th spectral band of $L_V$ satisfies
\[
[(k-1)^2 + \|V\|_\infty, k^2 - \|V\|_\infty]
\subseteq
B_k
\subseteq
[(k-1)^2 - \|V\|_\infty, k^2 + \|V\|_\infty]
\]
by general eigenvalue perturbation theory \cite{Kato1980:PertTh}. Then, it follows that the length of the $k$th band of $L_V$, grows approximately linearly in $k$. This can be viewed as one instance of a tendency for the spectrum of $L_V$ to thicken in the high-energy region (which is not present in the discrete case). In view of this, it is quite surprising that one can beat this tendency and prove the following result, due to Damanik, Fillman, and Lukic:

\begin{theorem} \label{t:DFL}
There is a Baire-generic subset $Z \subseteq \LP(\bbR)$ with the property that $\sigma(L_{\lambda V})$ is an {\rm(}unbounded{\rm)} Cantor set of zero Lebesgue measure for all $V \in Z$ and all $\lambda > 0$. There is a dense subset $H \subseteq \LP(\bbR)$ with the property that $\sigma(L_{\lambda V})$ is an {\rm(}unbounded{\rm)} Cantor set of zero Hausdorff dimension for all $V \in H$ and all $\lambda > 0$. Moreover, for all $V \in Z$ and all $V \in H$, the spectral type of $L_V$ is purely singular continuous.
\end{theorem}

The proof of Theorem~\ref{t:DFL} parallels that of \ref{t.lp.cantorspec.generic}, except that one is only able to prove suitable measure estimates in compact energy windows, since, as observed above, the Lebesgue measure of the bands grows linearly in the band label. Thus, one has to expand the compact sets on which one has effective measure estimates on the spectrum in a way that misses the lengthening effect of the high-energy region. See \cite{DFL} for more details.

\subsection{CMV Matrices and Quantum Walks}

CMV matrices arise naturally in the study of orthogonal polynomials on the unit circle (OPUC) and are universal within the class of unitary operators of spectral multiplicity one. Concretely, one starts with $\mu$, a Borel probability meaure supported on the unit circle which does not admit a support having only finitely many points. A \emph{CMV matrix} is a five-diagonal semi-infinite matrix that is determined by a sequence of \emph{Verblunsky coefficients} $\{\alpha_n\}_{n \in \bbZ_+} \subset \bbD$; these coefficients arise as the recursion coefficients of the orthogonal polynomials associated $\mu$. In terms of $\alpha_n$ and the derived quantities $\rho_n = \left( 1 - |\alpha_n|^2 \right)^{1/2}$, the CMV matrix takes the form
\begin{equation} \label{def:cmv}
\small
\mathcal C
=
\begin{bmatrix}
\overline{\alpha_0} & \overline{\alpha_1}\rho_0 & \rho_1\rho_0 &&& & \\
\rho_0 & -\overline{\alpha_1}\alpha_0 & -\rho_1 \alpha_0 &&& & \\
& \overline{\alpha_2}\rho_1 & -\overline{\alpha_2}\alpha_1 & \overline{\alpha_3} \rho_2 & \rho_3\rho_2 & & \\
& \rho_2\rho_1 & -\rho_2\alpha_1 & -\overline{\alpha_3}\alpha_2 & -\rho_3\alpha_2 &  &  \\
&&& \overline{\alpha_4} \rho_3 & -\overline{\alpha_4}\alpha_3 & \overline{\alpha_5}\rho_4 & \rho_5\rho_4 \\
&&& \rho_4\rho_3 & -\rho_4\alpha_3 & -\overline{\alpha_5}\alpha_4 & -\rho_5 \alpha_4  \\
&&&& \ddots & \ddots &  \ddots
\end{bmatrix}.
\end{equation}
This matrix defines a unitary operator in $\ell^2(\bbZ_+)$, and the spectral measure corresponding to $\mathcal C$ and the vector $\delta_0$ is given by $\mu$. This sets up a one-to-one correspondence between measures $\mu$ and coefficient sequences $\{\alpha_n\}_{n \in \bbZ_+}$, which has been extensively studied in recent years, mainly due to the infusion of ideas from Simon's monographs \cite{S1, S2}.

Similarly, an \emph{extended CMV matrix} is a unitary operator on $\ell^2(\bbZ)$ defined by a bi-infinite sequence $\{\alpha_n\}_{n \in \bbZ} \subset \bbD$ in an analogous way:
\begin{equation} \label{def:extcmv}
\small
\mathcal E
=
\begin{bmatrix}
\ddots & \ddots & \ddots &&&&&  \\
\overline{\alpha_0}\rho_{-1} & -\overline{\alpha_0}\alpha_{-1} & \overline{\alpha_1}\rho_0 & \rho_1\rho_0 &&& & \\
\rho_0\rho_{-1} & -\rho_0\alpha_{-1} & -\overline{\alpha_1}\alpha_0 & -\rho_1 \alpha_0 &&& & \\
&  & \overline{\alpha_2}\rho_1 & -\overline{\alpha_2}\alpha_1 & \overline{\alpha_3} \rho_2 & \rho_3\rho_2 & & \\
& & \rho_2\rho_1 & -\rho_2\alpha_1 & -\overline{\alpha_3}\alpha_2 & -\rho_3\alpha_2 &  &  \\
& &&& \overline{\alpha_4} \rho_3 & -\overline{\alpha_4}\alpha_3 & \overline{\alpha_5}\rho_4 & \rho_5\rho_4 \\
& &&& \rho_4\rho_3 & -\rho_4\alpha_3 & -\overline{\alpha_5}\alpha_4 & -\rho_5 \alpha_4  \\
& &&&& \ddots & \ddots &  \ddots
\end{bmatrix}.
\end{equation}
From the point of view of orthogonal polynomials, the study of $\mathcal C$ is more natural; however, when the Verblunsky coefficients are generated by an invertible ergodic map (such as a minimal translation of a compact abelian group as in the present paper), the study of $\mathcal E$ is more natural.

CMV matrices also arise in a natural fashion in the study of 1-dimensional coined quantum walks. A 1D quantum walk on $\bbZ$ is a quantum mechanical analog of a classical random walk, with two twists:
\begin{itemize}
\item The walker has an internal degree of freedom (called ``spin'') that influences her probability of hopping to the left or to the right.
\item The walker may exist as a superposition of pure states, rather than being fully localized.
\end{itemize}
The relevant state space is $\mathcal{H}_{\mathrm{QW}} = \ell^2(\bbZ)\otimes\bbC^2$; the $\ell^2(\bbZ)$ component captures the spatial position of the walker, while the $\bbC^2$ variable captures her spin. We will denote the pure states as $\delta_n^\pm = \delta_n \otimes e_\pm$, where $\{e_+,e_-\}$ denotes the usual basis of $\bbC^2$. Each quantum coin is then a superposition of two ``classical coins'', so the quantum walk is parameterized by a sequence of $2\times 2$ unitaries:
\[
Q_n
=
\begin{bmatrix}
q_n^{11} & q_n^{12} \\
q_n^{21} & q_n^{22}
\end{bmatrix}.
\]
As time advances one unit forward, the quantum walk update rule acts as follows on pure states:
\begin{align*}
U \delta_n^+
& =
q_n^{11} \delta_{n+1}^+ + q_n^{21} \delta_{n-1}^- \\
U \delta_n^-
& =
q_n^{12} \delta_{n+1}^+ + q_n^{22} \delta_{n-1}^-.
\end{align*}
Due to a seminal paper of Cantero, Gr\"unbaum, Moral, and Vel\'azquez, we know that the unitary update rule $U$ is unitarily equivalent to a CMV matrix \cite{CGMV}. Thus, any and all tools relevant to the study of CMV matrices enter the game and can be used to study 1D quantum walks.

Most of the results from the previous sections have CMV analogs (and hence analogs for quantum walks as well) -- however, owing to the more complicated structure of the CMV matrix compared to a Schr\"odinger operator, the proofs are generally more involved; see, e.g., \cite{FO2017,FOV,Ong2012}.

However, there is one notable exception to the previous remark: it is not yet known that CMV matrices having pure point spectrum are dense in the space of all limit-periodic CMV matrices. To spell things out more carefully, the Kunz--Souillard approach to localization is an essential ingredient in the arguments that proved Theorem~\ref{t.lppp1}. However, there is not a satisfactory version of the Kunz--Souillard localization proof for the CMV operators. In addition, to the best of our knowledge, no one has worked out an analog of P\"oschel's KAM scheme in the CMV setting. We would regard resolutions to either of these issues as very interesting results.

\subsection{Multidimensional Discrete Operators}

Given a bounded potential $V:\bbZ^d \to \bbR$, the discrete Schr\"odinger operator $H_V$ acts in $\ell^2(\bbZ^d)$ via
\[
[H_V u]_{\mathbf{n}}
=
V_{\mathbf{n}} u_{\mathbf{n}} + \sum_{\|\mathbf{m} - \mathbf{n}\|_1= 1} u_{\mathbf{m}}.
\]
 Given $\fp = (p_1,p_2,\ldots,p_d) \in \bbZ_+^d$, we say that a potential $V$ is $\fp$-periodic if
\[
V_{\mathbf{n}+ p_j \mathbf{e}_j}
=
V_{\mathbf{n}}
\text{ for all } \mathbf{n} \in \bbZ^d \text{ and all } 1 \le j \le d.
\]

\begin{theorem} \label{t:discBSC}
Suppose $d \geq 2$. For all $\fp = (p_1,p_2,\ldots,p_d) \in \bbZ_+^d$, there is a constant $C = C_\fp > 0$ such that the following holds true.
\begin{itemize}
\item If $V:\bbZ^d \to \bbR$ is $\fp$-periodic and $\|V\|_\infty  \leq C$, then $\sigma(H_V)$ consists of at most two connected components.
\item If at least one entry of $\fp$ is odd, $V$ is $\fp$-periodic, and $\|V\|_\infty  \leq C$, then $\sigma(H_V)$ consists of a single interval.
\end{itemize}
\end{theorem}
Theorem~\ref{t:discBSC} was proved first in the special case $d=2$ when $\gcd(p_1,p_2) = 1$ by Kr\"uger \cite{KrugPreprint} (note that at least of of $p_1$ and $p_2$ must be odd in this case). This was generalized to all periods in $d=2$ by Embree and Fillman \cite{EmbFil2018} and to all $d \geq 2$ by Han and Jitomirskaya \cite{HJ}. One should view this result as a discrete analog of the Bethe--Sommerfeld conjecture:

\begin{theorem} \label{t:BSC}
If $V:\bbR^d \to \bbR$ is periodic, then $\sigma(-\nabla^2 + V)$ has only finitely many gaps.
\end{theorem}

Theorem~\ref{t:BSC} has a rich history with contributions from many authors, including (but certainly not limited to) \cite{HelMoh98,Karp97,PopSkr81,Skr79,Skr84,Skr85,Vel88}, and culminating in the paper of Parnovskii \cite{Parn2008AHP}.

As a consequence of Theorem~\ref{t:discBSC}, one can show that small limit-periodic operators in $\ell^2(\bbZ^d)$ for $d \geq 2$ also will have spectra comprising only one or two intervals.

\begin{coro}
Let $d\geq 2$ and $\fp_n \in \bbZ_+^d$ be such that the $j$th coordinate of $\fp_n$ divides the $j$th coordinate of $\fp_{n+1}$ for all $1 \le j \le d$ and all $n \in \bbZ_+$. Then, there exists a sequence $\delta_n > 0$ with the following property: if $V_n$ is $\fp_n$-periodic and $\|V_n\|_\infty \leq \delta_n$ for each $n$, then the limit-periodic potential
\[
V_{\mathrm{lp}}
=
\sum_{n=1}^\infty V_n
\]
is such that $\sigma(H_{V_{\mathrm{lp}}})$ consists of at most two connected components. If at least one component of $\fp_n$ is odd for every $n \in \bbZ_+$, then $\sigma(H_{V_{\mathrm{lp}}})$ is a single interval.
\end{coro}

In particular, it is substantially harder to produce Cantor spectrum in higher dimensions than in dimension one.

\section{Open Problems} \label{sec:problems}

In this section we describe a number of open problems that are suggested by the existing results. Our first set of questions can be summarized by asking whether or not any limit-periodic operator ever experiences any sort of phase transition. Concretely, whenever a result is known about a limit-periodic potential $V$, it is known that the spectral type of $H_V$ is pure. Thus far, any results that give information about the hull of $V$ show that the spectral type is constant on the hull of $V$. Finally, whenever one is able to prove a result about the one-parameter family $\{\Delta + \lambda V\}_{\lambda > 0}$, the spectral type of $H_{\lambda V}$ does not vary with $\lambda > 0$. This leads us to our first three questions:

\begin{question}
Is the spectral type of a limit-periodic operator always pure?
\end{question}

\begin{question}
Does the spectral type ever change as one passes to other elements of the hull?
\end{question}

\begin{question}
Considering a limit-periodic Schr\"odinger operator and replacing the potential by a non-trivial multiple of it (i.e., by varying the coupling constant), can the spectral type ever change?
\end{question}

\begin{question}
We know that the occurrence of Cantor spectra is a generic phenomenon in the limit-periodic universe. What about the failure of Cantor spectrum? We know that it is possible (certainly for periodic cases, but also for non-periodic cases due to P\"oschel's work \cite{P83}). Is it a dense phenomenon among the non-periodic limit-periodic cases?
\end{question}

Let us discuss the general questions above, and others, in specific settings:

\begin{question}
Suppose $V$ is limit-periodic and $\sigma_{\ac}(H_V) \neq \emptyset$.
\begin{enumerate}
\item Is it true that $H_V$ has purely a.c.\ spectrum?
\item Is it true that $H_{\lambda V}$ has purely a.c.\ spectrum for all $\lambda$?
\item Is it true that $H_W$ has purely a.c.\ spectrum for all $W \in \hull(V)$?
\end{enumerate}
\end{question}

\begin{question}
Consider a limit-periodic Schr\"odinger operator with pure point spectrum that arises via the generalized Kunz-Souillard approach presented in Section~\ref{s.pp}. Pass to a different element of the hull and/or vary the coupling constant. Does the spectral type remain pure point?
\end{question}

\begin{question}
In the setting of the previous problem (limit-periodic operators obtained via the generalized Kunz-Souillard approach), can one show that the eigenfunctions \emph{cannot} decay exponentially? In other words, can one show that the Lyapunov exponent vanishes throughout the spectrum? This would be a new phenomenon within the class of ergodic Schr\"odinger operators: eigenvalues inside the set of energies where the Lyapunov exponent vanishes.
\end{question}

\begin{question}
Consider a limit-periodic Schr\"odinger operator with pure point spectrum that arises via the P\"oschel approach presented in Section~\ref{s.pp}. Recall that this puts us in the regime of large potentials. Vary the coupling constant, and consider in particular the case where it is chosen to be small. Does the spectral type remain pure point? Does is remain constant across the hull? If so, is this due to the persistence of uniform localization?
\end{question}

\begin{question}
In the general P\"oschel approach, can one find additional examples that display a variety of features? P\"oschel provided two examples in \cite{P83}, showing that the spectrum can be a Cantor set or it can have no gaps at all. What about other sets (sets with finitely many gaps; sets that are non-trivial unions of Cantor sets and finite unions of non-degenerate intervals; sets that are Cantorvals)?
\end{question}

Speaking of the P\"oschel and Kunz-Souillard approaches, it would be interesting to work out their analogs for CMV matrices.

\begin{question}
Can the approach to limit-periodic Schr\"odinger operators from P\"oschel's paper \cite{P83} be carried over to CMV matrices?
\end{question}

\begin{question}
Can the Kunz-Souillard approach to random Schr\"odinger operators from \cite{KS80} be carried over to CMV matrices? If so, is there an extension of it analogous to \cite{DG16}, which then has similar consequences for almost-periodic CMV matrices?
\end{question}

\begin{appendix}

\section{Profinite Groups}\label{s.profinite}

Since the hull of a limit-periodic potential is a totally disconnected group, we will discuss some characteristics of totally disconnected groups. The goal of the appendix is to provide a short, self-contained proof of two results: that compact totally disconnected groups are profinite and that monothetic Cantor groups are procyclic. Let us recall that a \emph{topological group} is a Hausdorff topological space $G$ that is endowed with a group structure in such a way that the group operations (composition and inversion) are both continuous. For thorough treatments of totally disconnected groups, see \cite{ribezal, wilson}.

\subsection{Inverse Limits}

\begin{definition}
A \emph{partial order} on a set $I$ is a binary relation $\preceq$ with the following properties:
\begin{enumerate}
\item $i \preceq i$ for every $i \in I$.
\item $i = j$ whenever $i \preceq j$ and $j \preceq i$.
\item $i \preceq k$ whenever $i \preceq j$ and $j \preceq k$.
\end{enumerate}

\noindent Given a nonempty set $I$ equipped with a partial order $\preceq$, an \emph{inverse system} of topological groups over $(I,\preceq)$ is an ordered pair $( (G_i),(\varphi_i^j) )$, in which $G_i$ is a topological group for each $i \in I$ and $\varphi_i^j:G_j \to G_i$ is a continuous homomorphism whenever $i \preceq j$ such that
\begin{enumerate}
\item  $\varphi_i^i = \mathrm{id}_{G_i}$ for all $i \in I$
\item The $\varphi$'s are compatible in the sense that $\varphi_i^j \circ \varphi_j^k = \varphi_i^k$ whenever $i \preceq j \preceq k$. Equivalently, the following diagram commutes for every triple $i,j,k \in I$ for which $i \preceq j \preceq k$:
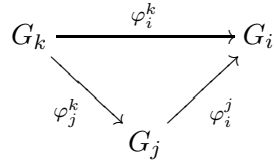
$$
\xymatrix{
G_k \ar[rd]_{\varphi_j^k} \ar[rr]^{\varphi_i^k}
&& G_i \\
& G_j \ar[ru]_{\varphi_i^j} &
}
$$
\end{enumerate}
An \emph{inverse limit} (or \emph{projective limit}) of an inverse system is a group $G$ together with continuous homomorphisms $\psi_i:G \to G_i$ with the following properties:
\begin{enumerate}
\item The maps $\psi_i$ are compatible with the $\varphi$'s in the sense that $\varphi_i^j \circ \psi_j = \psi_i$ whenever $i \preceq j$, i.e., the following diagram commutes:
$$
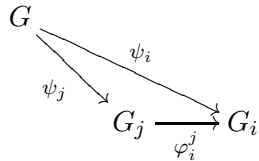
\xymatrix{
G \ar[rd]_{\psi_j} \ar[rrd]^{\psi_i}  & &\\
& G_j \ar[r]_{\varphi_i^j} & G_i
}
$$
\item The pair $(G,(\psi_i))$ is minimal with respect to property (1) in the following sense: if $G'$ and $\psi_i'$ are a group and a family of morphisms with $\varphi_i^j \circ \psi_j' = \psi_i'$ whenever $i \preceq j$, then there exists a unique continuous homomorphism $\theta:G' \to G$ such that $\psi_i' = \psi_i \circ \theta$ for every $i \in I$. Pictorially, the induced map $\theta$ is such that the following diagram commutes:
$$
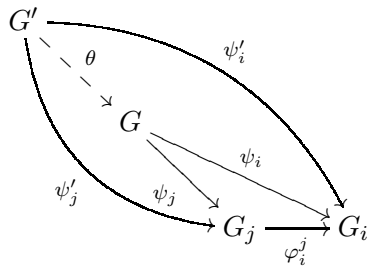
\xymatrix{
G'
\ar@{-->}[rd]^{\theta}
\ar@/^2pc/[rrrdd]^{\psi_i'}
\ar@/_2pc/[rrdd]_{\psi_j'}
&&& \\
&
G \ar[rrd]^{\psi_i} \ar[rd]_{\psi_j}
&& \\
&&
G_j
\ar[r]_{\varphi_i^j}
&
G_i
}
$$
\end{enumerate}
\end{definition}
\begin{theorem}
Any inverse system of topological groups has an inverse limit which is unique up to isomorphism.
\end{theorem}

\begin{proof}
Put
\begin{equation} \label{eq:invlim:stdconstr}
G
=
\left\{
(g_i)_{i \in I} : \varphi_i^j(g_j) = g_i \text{ whenever } i \preceq j
\right\}
\subseteq
\prod_{i \in I} G_i,
\end{equation}
and let $\psi_i:G \to G_i$ denote projection onto the $i$th coordinate. A simple diagram chase shows that $(G,\psi_i)$ satisfies the definition of an inverse limit. If $G_0$ is another inverse limit, then the $\theta$'s guaranteed by the universal mapping property furnish isomorphisms between $G$ and $G_0$.
\end{proof}

We denote the inverse limit of the inverse system $( (G_i), (\varphi_i^j) )$ by
$$
G = \varprojlim G_i =  \varprojlim(G_i,\varphi_{i}^j),
$$
and we typically use the first notation, even though $G$ clearly depends on the maps $\varphi_i^j$.

\begin{remark}
Of course, we have described the construction of the categorical inverse limit in the category whose objects are topological groups and whose morphisms are continuous homomorphisms. One could just as well speak of the inverse limit of rings, modules over a PID, or topological spaces, for example, but we will not have any need for more general inverse limits.
\end{remark}

\begin{example}
Let $I$ be any nonempty set, and equip $I$ with the trivial partial order where $\alpha \prec \alpha$ for any $\alpha \in I$ and $\alpha \not\prec \beta$ for every $\alpha \neq \beta$. Then, one can verify that
\[
\varprojlim G_i
=
\prod_{i \in I} G_i;
\]
note that there is no need to specify maps $\varphi_i^j$ in this case.
\medskip

Given a prime $p \in \bbZ_+$, the additive group of $p$-adic integers is defined to be the set of formal sums:
\[
\bbJ_p
=
\set{\sum_{j=0}^\infty a_j p^j : a_j \in \{0,1,\ldots,p-1\}},
\]
where the group operation is given by ``adding with carrying.'' One can realize $\bbJ_p$ as an inverse limit of cyclic groups. To see this, let $I = \bbZ_+$ with the usual order, and let $G_k = \bbZ_{p^k}$ for $k \in \bbZ_+$. Define
$$
\varphi_k^\ell\left(n + p^\ell \bbZ \right) = n + p^k \bbZ
$$
whenever $k \leq \ell$. The inverse limit of this inverse system is canonically isomorphic to $\bbJ_p$. Specifically, if we define $\psi_k: \bbJ_p \to \bbZ_{p^k}$ by
$$
\psi_k \left( \sum_{j = 0}^\infty a_j p^j \right)
=
\left(\sum_{j=0}^{k-1} a_j p^j \right) + p^k \bbZ,
$$
one can check that $\bbJ_p$ together with this family of maps satisfies the definition of an inverse limit.
\end{example}

\begin{definition}
We say that a group $G$ is \emph{profinite} if it can be written as the inverse limit of an inverse system of finite groups and \emph{procyclic} if it can be written as the inverse limit of an inverse system of cyclic groups.
\end{definition}

\begin{prop} \label{p:profinitegrps:td}
The inverse limit of a family of compact totally disconnected groups is itself compact and totally disconnected. In particular, profinite groups are compact, Hausdorff, and totally disconnected.
\end{prop}

\begin{proof}
The product of compact spaces is compact by Tychonoff's theorem. It is straightforward to check that $G$ defined in \eqref{eq:invlim:stdconstr} is closed with respect to the product topology, and hence compact. Since products and subspaces of totally disconnected (respectively Hausdorff) spaces are totally disconnected (respectively Hausdorff),  the proposition follows.
\end{proof}

The remainder of the appendix will focus on the converse of the preceding proposition; that is, any totally disconnected compact group is profinite. Indeed, such a group may be naturally identified with its so-called \emph{profinite completion}.

\begin{definition}
Let $G$ be a topological group, and consider the index set
$$
I
=
\mathcal N
:=
\{ N \subseteq G : N \text{ is an open normal subgroup of } G \},
$$
partially ordered by reverse inclusion (i.e., $H \preceq K$ if and only if $H \supseteq K$). For $H \supseteq K$, define $\varphi_H^K: G/K \to G/H$ denote the natural projection, i.e.,
$$
\varphi_H^K (gK) = gH.
$$
The inverse limit of this system is called the \emph{profinite completion} of $G$, denoted
$$
\hat G
=
\varprojlim G/H.
$$
There is a canonical map $\rho: G \to \hat G$ given by
\begin{equation} \label{eq:profcomp:imbed}
\rho(g)
=
(gN)_{N \in \mathcal N} \in \prod_{N \in \mathcal N} G/N.
\end{equation}
Notice that $\hat G$ may be trivial, and hence, $\rho$ need not be injective. For example, the only open subgroup of $\bbT$ is $\bbT$ itself, so the profinite completion of $\bbT$ is trivial.
\end{definition}

\begin{theorem} \label{t:profgrap:char}
A compact topological group $G$ is profinite if and only if it is totally disconnected. Indeed, any compact totally disconnected group is isomorphic to its profinite completion.
\end{theorem}

To prove this, we need a pair of preliminary results.

\begin{prop} \label{p:tdhausd:zd}
Any totally disconnected compact Hausdorff space is of topological dimension zero in the sense that its topology enjoys a neighborhood basis consisting of compact open sets.
\end{prop}

\begin{proof}
Suppose $X$ is a totally disconnected compact Hausdorff space, let $x \in X$ be given, and denote by $C = C(x)$ the intersection of all compact open sets containing $x$; notice that $C$ is necessarily closed, but it need not be open. First, we claim that $C = \{x\}$. If not, then total disconnectedness of $X$ implies that $C$ is disconnected, so it can be written as a disjoint union
$$
C = A \cup B,
$$
where $A$ and $B$ are disjoint, nonempty, closed subsets of $C$. Since $C$ is a closed subset of $X$, both $A$ and $B$ are closed subsets of $X$ as well. Since $X$ is compact and Hausdorff, we may choose disjoint open sets $U, V \subseteq X$ with $A \subseteq U$ and $B \subseteq V$. Since $C \subseteq U \cup V$, the \emph{complements} of the compact open sets containing $x$ comprise an open cover of $X \setminus (U \cup V)$, so, by using compactness to reduce to a finite subcover, we see that there are finitely many compact open sets $C_1,\ldots,C_n$ containing $x$ such that
$$
P
:=
\bigcap_{j=1}^n C_j
\subseteq
U \cup V.
$$
Of course, $P$ is itself compact and open. Notice that
$$
\overline{P \cap U}
\subseteq
P \cap \overline{U}
=
(U \cup V) \cap P \cap \overline{U}
=
P \cap U.
$$
Consequently, $P \cap U$ is compact, open, and contains no points of $B$. However, this contradicts the definition of $C$. Thus, $C(x) = \{x\}$.
\newline

Now, given an open set $V \subseteq X$ containing $x$, $C = \{x\}$ implies that the complements of the compact open sets containing $x$ comprise an open cover of $X \setminus V$. In particular, we may choose a finite collection $K_1, \cdots, K_n$ of compact open subsets of $X$ with
$$
x \in \bigcap_{j=1}^n K_j \subseteq V.
$$
As the intersection of finitely many compact open sets is itself compact and open, the proposition follows.
\end{proof}

\begin{prop} \label{p:tdgroup:zd} \label{p.small.open.subgroups}
Suppose $G$ is a totally disconnected compact group. Then the collection
$$
\mathcal N
=
\{N \subseteq G : N \text{ is an open normal subgroup of } G\}
$$
is a neighborhood basis for the topology of $G$ at the identity, $e$. In particular,
\begin{equation} \label{eq:openSubgrpIntersection}
\bigcap_{N \in \mathcal{N}}N
=
\{e\}.
\end{equation}
\end{prop}

\begin{proof}
By Proposition~\ref{p:tdhausd:zd}, it suffices to prove that any compact open set containing $e$ contains an open normal subgroup of $G$. To that end, suppose that $K \ni e$ is compact and open.
\newline

\noindent \textbf{Claim.} There is a symmetric neighborhood $V$ of $e$ for which $K \cdot V \subseteq K$.
\newline

For each $g \in K$, openness of $K$ and continuity of multiplication in $G$ implies that there exists an open set $U = U_g$ containing $e$ with $U_g^2 \subseteq g^{-1} \cdot K$. Here, we denote
\[
U^2
=
\set{uv : u \in U \text{ and } v \in U}.
\]
By replacing $U$ with $U \cap U^{-1}$ we may assume without loss of generality that $U_g = U_g^{-1}$ for each $g$. Since $K$ is compact, one has
$$
K \subseteq \bigcup_{j=1}^m (g_j \cdot U_{g_j})
$$
for some finite collection $g_1,\ldots,g_m \in K$. The open set $V = \bigcap_{j=1}^m U_{g_j}$ is a symmetric neighborhood of the identity with $VK \subseteq K$.
\newline

Let us see how the proposition follows from the claim. Since $V$ is symmetric, the subgroup it generates is given by
$$
N
=
\bigcup_{n = 1}^\infty V^n.
$$
Since $V$ is open, $N$ is clearly open. Moreover, $K \cdot V \subseteq K$ implies that $N \subseteq K$. By compactness, $N$ has finite index and hence there are only finitely many subgroups of $G$ conjugate to $N$. The intersection of these (finitely many) conjugate subgroups is an open normal subgroup of $G$ contained in $K$.

\end{proof}

\begin{proof}[Sketch of Proof of Theorem~\ref{t:profgrap:char}]
Let $G$ be a totally disconnected compact group and let $\rho:G \to \hat{G}$ be the homomorphism defined in \eqref{eq:profcomp:imbed}.
\medskip

\noindent \textbf{Step 1.} The image of $\rho$ is dense in $\hat{G}$.

Suppose $U \subseteq \hat G$ is open and nonempty. Without loss, we may assume that $U$ is of the form
$$
U
=
\hat G \cap \left( \prod_{N \in \mathcal N} U_N \right),
$$
where $U_N = G/N$ for all but finitely many $N \in \mathcal N$. If $N_1,\ldots,N_k$ denote the exceptional elements of $\mathcal N$, put
$$
\widetilde N
=
\bigcap_{j=1}^k N_j
$$
and choose $(g_N \cdot N)_{N \in \mathcal N} \in U$. By the compatibility conditions, $g_{\widetilde N} N_j = g_{N_j} N_j$ for each $1 \leq j \leq k$, whence $\rho(g_{\widetilde N}) \in U$.
\medskip

\noindent \textbf{Step 2.} $\rho$ is surjective.

By Step~1, the image of $\rho$ is dense. On the other hand, it must be closed by compactness of $G$ and continuity of $\rho$.
\medskip

\noindent \textbf{Step 3.}  $\rho$ is injective. This is immediate from \eqref{eq:openSubgrpIntersection}.

Thus, $\rho$ is a continuous bijective homomorphism. To see that $\rho^{-1}$ is continuous, simply apply the Open Mapping Theorem (see, e.g., \cite[Theorem~3, Chapter~1]{Morris1977:TopGrps})
\end{proof}

\begin{coro}
A compact monothetic totally disconnected group is procyclic.
\end{coro}

\begin{proof}
Suppose $G$ is compact, monothetic and totally disconnected.  Let $\alpha$ denote a generator of a dense cyclic subgroup. If $N$ is any open subgroup of $G$, then $G/N$ is cyclic -- indeed, the coset $\alpha + N$ can be seen to generate $G/N$ by minimality. Since $G$ is isomorphic to the inverse limit of such quotients by Theorem~\ref{t:profgrap:char}, $G$ is procyclic.
\end{proof}

\end{appendix}

\end{document}